\documentclass[11 pt]{amsart}

\usepackage[margin=2.5cm]{geometry}
\usepackage{amsfonts,graphicx,enumerate}
\usepackage{amssymb, mathrsfs}
\usepackage{amsmath}
\usepackage{latexsym}
\usepackage{mathrsfs}
\usepackage{amsthm}
\usepackage{verbatim}
\usepackage{amscd}
\usepackage{mathtools}
\usepackage[pagebackref]{hyperref}
\hypersetup{bookmarksdepth=2}

\usepackage{comment}
\usepackage{framed}
\usepackage{color}
\definecolor{shadecolor}{gray}{0.875}

\newtheorem{thrm}{Theorem}[section]
\newtheorem{lem}[thrm]{Lemma}
\newtheorem{cor}[thrm]{Corollary}
\newtheorem{prop}[thrm]{Proposition}
\newtheorem{conj}[thrm]{Conjecture}

\newtheorem{alphthrm}{Theorem}

\newtheorem{step}{Step}[thrm]

\theoremstyle{definition}
\newtheorem{defn}[thrm]{Definition}

\newtheorem{ex}[thrm]{Example}
\newtheorem{rmk}[thrm]{Remark}

\newtheorem{notn}[thrm]{Notation}

\DeclareMathOperator{\st}{\,\bigm\vert\,}

\DeclareMathOperator{\Sym}{S}

\DeclareMathOperator{\bl}{Bl}

\DeclareMathOperator{\sh}{\varepsilon\,}

\DeclareMathOperator{\Eff}{\overline{Eff}}
\DeclareMathOperator{\Mov}{\overline{Mov}}
\DeclareMathOperator{\Nef}{{Nef}}

\DeclareMathOperator{\Supp}{Supp\,}

\DeclareMathOperator{\mult}{mult}

\DeclareMathOperator{\Spec}{Spec}
\DeclareMathOperator{\bs}{Bs}

\let\cal\mathcal
\let\frak\mathfrak
\let\bb\mathbb
\let\scr\mathscr

\usepackage[arrow,matrix,cmtip]{xy}

\newcommand{\factor}[2]{\left. \raise 1pt\hbox{\ensuremath{#1}} \right/
        \hskip -2pt\raise -3pt\hbox{\ensuremath{#2}}}

\numberwithin{equation}{thrm}

\begin{document}

\title{Seshadri constants for vector bundles}
\author{Mihai Fulger}
\address{Department of Mathematics, University of Connecticut, Storrs, CT 06269-1009, USA}
\address{Institute of Mathematics of the Romanian Academy, P. O. Box 1-764, RO-014700,
Bucharest, Romania}
\email{mihai.fulger@uconn.edu}
\thanks{The first author was partially supported by the Simons Foundation
Collaboration Grant 579353.}
\author{Takumi Murayama}
\address{Department of Mathematics\\University of Michigan\\
Ann Arbor, MI 48109-1043, USA}
\email{takumim@umich.edu}
\thanks{The second author was partially supported by the National Science
Foundation under Grant No.\ DMS-1501461.}

\begin{abstract}
We introduce Seshadri constants for line bundles in a relative setting. 
They generalize the classical Seshadri constants of line bundles on projective
varieties and their extension to vector bundles studied by
Beltrametti--Schneider--Sommese and Hacon.
There are similarities to the classical theory. In particular, we give a Seshadri-type ampleness criterion, and we relate Seshadri constants to jet separation and to asymptotic base loci.

We give three applications of our new version of Seshadri constants.
First, a celebrated result of Mori can be restated as saying that any Fano manifold whose tangent bundle has positive Seshadri constant at a point is isomorphic to a projective space. We conjecture that the Fano condition can be removed. 
Among other results in this direction, we prove the conjecture for surfaces.
Second, we restate a
classical conjecture on the nef cone of self-products of curves 
in terms of semistability of higher conormal sheaves, which we use to identify new nef classes on self-products of curves. 
Third, we prove that our Seshadri constants can be used to control separation of
jets for direct images of pluricanonical bundles, in the spirit of a relative
Fujita-type conjecture of Popa and Schnell.
\end{abstract}

\maketitle

\tableofcontents

\section{Introduction}
Let $X$ be a projective scheme over an algebraically closed field, and let $\cal L$
be an ample line bundle on $X$.
In \cite[Section 6]{dem}, Demailly defined the \emph{Seshadri constant}
$\sh(\cal L;x)$ of $\cal L$ at a closed point $x \in X$ by
\[
  \sh(\cal L;x) \coloneqq   \sup\bigl\{ t \in \bb R_{\ge0}\st\pi^*c_1(\cal L)-tE\ \text{is
  nef}\bigr\},
\]
where $\pi$ is the blow-up of $X$ at $x$ with exceptional divisor $E$.
Seshadri constants have attracted much attention as
interesting invariants that capture subtle geometric properties of both $X$ and
$\cal L$; see \cite[Chapter 5]{laz04} and \cite{primer}.
In higher rank, a version of Seshadri constants for
ample vector bundles (of arbitrary rank) appears implicitly in work of Beltrametti--Schneider--Sommese \cite{BSS93,BSS96}, and has been further studied by Hacon \cite{hacon}.

\par In this paper, we define a new version of Seshadri constants for line
bundles in a relative setting, generalizing both Demailly's and Hacon's
definitions.
One advantage of this version is that it does not impose any global
positivity conditions on the line bundle or vector bundle in question. We refer
to \S\ref{section:definition} for the precise definition.
In the case of vector bundles $\cal V$ on $X$, loosely speaking 
\[\sh(\cal V;x)\coloneqq  \sup\biggl\{t\in\bb R \biggm\vert
  \begin{tabular}{@{}c@{}}
    $\pi^*\cal V\langle-tE\rangle$ is nef on curves that\\
    meet $E$ properly in at least one point
  \end{tabular}\biggr\}.\]
Many of the classical properties of Seshadri constants generalize to our new version.
\begin{enumerate}[(1)]
  \item A Seshadri ampleness criterion holds (Theorem \ref{seshadriample}),
    generalizing \cite[Theorem 1.4.13]{laz04}.
  \item We have homogeneity for vector bundles in the
  sense that $\sh(\Sym^m\cal V;x)=m\cdot\sh(\cal V;x)$ (Lemma \ref{lem:homogeneous}) and $\sh(\bigotimes^m\cal V;x)=m\cdot\sh(\cal V;x)$ (Proposition \ref{lem:tensorproducts}). The case of line bundles is trivial. 
\item For ample vector bundles, the Seshadri constant measures asymptotic jet
  separation (Theorem \ref{thrm:jetsep}). This generalizes Demailly's result
  \cite[Theorem 6.4]{dem}, and is new even for Seshadri constants of line
  bundles at singular points.
\item The Seshadri constants satisfy semicontinuity in both a convex geometric
  sense and in a variational sense (see \S\ref{section:semi}).
\item For nef vector bundles $\cal V$, the locus $\{x\in X \mid
  \sh(\cal V;x)=0\}$ coincides with the non-ample locus ${\bb B}_+(\cal V)$
  (Proposition \ref{prop:baseseshadrivanish}). The line bundle case, due to Nakamaye, can be found in \cite{nakamaye,ELMNP}.
\item For big and nef vector bundles, lower bounds on Seshadri constants lead to
  lower bounds on the order of jet separation for adjoint bundles (Proposition \ref{prop:jetsepbound}). These generalize the rank 1 case in \cite[Proposition 6.8]{dem}.
\end{enumerate}
\subsection{Examples}
We describe our version of the Seshadri constant in some examples.
\begin{ex}[Vector bundles on curves] In \cite[Theorem 3.1]{hacon}, Hacon proves that if $\cal V$ is a vector bundle on a smooth complex projective curve $X$,
then 
\[\sh(\cal V;x)=\mu_{\rm min}(\cal V)\]
for all $x\in X$. Here, $\mu_{\rm min}(\cal V)$ is the smallest slope in the Harder--Narasimhan filtration of $\cal V$.
We prove a similar description in positive characteristic by replacing $\cal V$
with iterated Frobenius pullbacks of $\cal V$; 
see Example
\ref{ex:curves}.

This example is fundamental to the development of the theory. 
It helps reduce many results to the case where $X$ is a smooth projective curve,
where they are significantly easier.
\end{ex}

\begin{ex}[Toric bundles] In \cite[Proposition 3.2]{hmp}, Hering, Musta\c{t}\u{a}, and Payne compute Seshadri constants for \emph{nef} toric bundles $\cal V$ on smooth toric varieties at the torus invariant points $x_{\sigma}$. 
  They show that $\sh(\cal V;x_\sigma)$ is the 
 smallest degree of any summand of the restrictions of $\cal V$ to the invariant $\bb P^1$'s through $x_{\sigma}$.
\end{ex}
\begin{ex}[Tangent bundle of homogeneous spaces; see Examples \ref{ex:tangentpn}
  and \ref{ex:tangenthomog}]\label{ex:introhomog}
  Let $X$ be a homogeneous space (e.g., a rational homogeneous space, or abelian
  variety). Then,
  \[
    \sh(TX;x)= \begin{cases*}
      2 & if $X \simeq \bb P^1$;\\
      1 & if $X \simeq \bb P^n$, where $n \ge 2$;\\
      0 & otherwise.
    \end{cases*}
  \]
\end{ex}

\subsection{Applications}
We now describe applications of our new version of Seshadri constants.
Our first application gives new characterizations of projective space.
A celebrated result of Mori \cite{mori} states that if $X$ is an
$n$-fold with ample tangent bundle, then $X \simeq \bb P^n$.
Thus, $\bb P^n$ is the only projective manifold with ``very positive'' tangent bundle. 
It is natural to ask if any weaker positivity conditions on $TX$ still ensure that $X\simeq\bb P^n$.
Example \ref{ex:introhomog} shows $\bb P^n$ is the only homogeneous space whose tangent bundle has positive Seshadri constant at one point.
The following results says that even without assuming that $X$ is a homogeneous
space, this condition implies $X \simeq \bb P^n$ in many cases.

\begin{alphthrm}[see Proposition \ref{prop:Fanos} and Corollary \ref{cor:charpnsurfaces}]\label{thrm:charpn}
  Let $X$ be a smooth projective variety of dimension $n$ over an algebraically
  closed field $k$.
  Suppose $\sh(TX;x_0) > 0$ for some closed point $x_0 \in X$, and
  suppose that one of the following conditions holds:
  \begin{enumerate}
    \item[\textup{(1)}] $X$ is Fano;
    \item[\textup{(2)}] $\operatorname{char} k = 0$ and $x_0$ is general in the
      sense of \cite[Notation 2.2]{kebekus}; or
    \item[\textup{(3)}] $\dim X = 2$.
  \end{enumerate}
  Then, $X$ is isomorphic to the $n$-dimensional projective space $\bb P^n$.
\end{alphthrm}
The theorem is also inspired by similar results for Seshadri constants of divisors
due to Bauer--Szemberg \cite{BS09},
Liu--Zhuang \cite{LZ18}, the second author \cite{Mur18}, and Zhuang
\cite{Zhu17char0,Zhu17charp}. They find characterizations of projective spaces in terms of lower bounds
of the form $\sh(-K_X;x_0) > n$.
We conjecture that Theorem \ref{thrm:charpn} holds without any of the
additional assumptions (1)--(3).

The proofs for (1) and (2) follow easily from Mori's work and from
\cite{cmsb02}, respectively.
For (3), we show that the condition $\sh(TX;x_0)>0$ is preserved by smooth blow-downs away from $x_0$ (in arbitrary dimension). 
We then use the Enriques classification of minimal surfaces. 
\medskip
\par Our second application uses our new version of Seshadri constants to study
the nef cone of products of curves.
Recall the following conjecture:
\begin{conj}[see {\cite[Remark 1.5.10]{laz04}}]\label{conj:prodcurvesintro}
  Let $C$ be a smooth projective curve of genus $g$ over $\bb C$.
  Denote by $f_1,f_2$ (resp.\ $\delta$) the classes of the fibers of the
  projections (resp.\ the class of the diagonal) in $C \times C$.
  Then, we have
  \[
    ( \sqrt{g} + 1)(f_1+f_2) - \delta \in \operatorname{Nef}^1(C \times
    C)
  \]
  if $g$ is sufficiently large and $C$ is very general.
\end{conj}
The self intersection of $( \sqrt{g} + 1)(f_1+f_2) - \delta$ is zero, just like in the famous Nagata conjecture.
In fact, \cite{cknagata,Ross} prove that the Nagata conjecture implies Conjecture \ref{conj:prodcurvesintro}.
The best known result here is due to Kouvidakis \cite[Theorem 2]{kouvidakis} (see also \cite[Corollary 1.5.9]{laz04}), who shows that
\[\left(\frac g{\lfloor \sqrt g\rfloor}+1\right)(f_1+f_2)-\delta\in\Nef^1(C\times C).\]
In particular, the conjecture holds when $g$ is a perfect square 
(just like the Nagata conjecture). 
For arbitrary $a>1$, it also makes sense to consider the non-symmetric divisors with zero self-intersection and ask:  
\[\text{For }a>1\text{, is the class }af_1+\Bigl(1+\frac g{a-1}\Bigr)f_2-\delta\text{ in }\Nef^1(C\times C)\,?\]
The best known result here appears to be due to Rabindranath
\cite[Proposition 3.2]{ashwath}. 
He adapts an idea of Vojta to prove that
\begin{equation}\label{eq:vojtarabindranath}
  af_1+\left(1+\frac{2g}{a-1+\sqrt{(a-1)^2-4g(g-1)}}\right)f_2-\delta\in\Nef^1(C\times C).
\end{equation}
We prove in Theorem \ref{thrm:prodcurves}.(i) that Conjecture \ref{conj:prodcurvesintro} 
and its generalization to non-symmetric classes can be reduced to a
statement about semistability of higher conormal sheaves in the spirit of
\cite{elstable}.
We then show the following:

\begin{alphthrm}[see Theorem \ref{thrm:prodcurves}.(ii)]
  Let $C$ be a general smooth projective curve of genus $g\geq 3$ over 
$\bb C$.
  Denote by $f_1,f_2$ (resp.\ $\delta$) the classes of the fibers of the
  projections (resp.\ the class of the diagonal in $C \times C$).
  Then, we have
  \[
    df_1 + \biggl( 1 + \frac{g}{d-g} \biggr) f_2 - \delta \in
    \operatorname{Nef}^1(C \times C)
  \]
  for every integer $d \ge \lfloor 3g/2\rfloor+3$.
\end{alphthrm}
When $d<2g$, these divisors are better than the known bounds described in \eqref{eq:vojtarabindranath} due to Vojta and Rabindranath.
For large $d$, they are close to the conjectural 
$df_1+\bigl(1+\frac g{d-1}\bigr)f_2-\delta$.
\medskip
\par Our last application shows that our version of Seshadri constants can be
used to control jet separation of direct images of pluricanonical sheaves, in the
spirit of a relative Fujita-type conjecture of Popa and Schnell
\cite[Conjecture 1.3]{popaschnell}.
This statement extends a result of Dutta and the second author
\cite[Theorem A]{DuttaMurayama} to vector bundles of higher rank, and to
higher-order jets.
See Theorem \ref{thrm:dmthmaanalogue} and Corollary \ref{cor:dmthmaanalogue} for
effective statements that do not mention $\sh(\cal V;x)$.
\begin{alphthrm}[see Theorem \ref{thrm:dmthmaanalogue}]
  Let $f\colon Y \to X$ be a surjective morphism of complex projective
  varieties, where $X$ is of dimension $n$.
  Let $(Y,\Delta)$ be a log canonical $\mathbb{R}$-pair and let $\mathcal{V}$ be
  a locally free sheaf of finite rank $r \ge 1$ on $X$ such that $\cal O_{\bb
  P(\cal V)}(1)$ is big and nef.
  Consider a Cartier divisor $P$ on $Y$ such that $P \sim_{\mathbb{R}}
  k(K_Y+\Delta)$ for some integer $k \ge 1$, and consider a general smooth
  closed point $x \in X \setminus \mathbb{B}_+(\mathcal{V})$.
  If $\sh(\mathcal{V};x) > k \cdot \frac{n+s}{m+k(r-1)+1}$,
  then the sheaf
  \[
    f_*\mathcal{O}_Y(P) \otimes_{\mathcal{O}_X} \Sym^m\mathcal{V}
    \otimes_{\mathcal{O}_X} (\det \mathcal{V})^{\otimes k}
  \]
  separates $s$-jets at $x$.
\end{alphthrm}

\subsection{Moving Seshadri constants}
For nef locally free sheaves $\cal V$, we can interpret 
$\bb B_+(\cal V)$ as the locus where Seshadri constants vanish.
For ample locally free sheaves $\cal V$, 
the asymptotic order of jet separation at $x$ is in fact equal to
$\sh(\cal V;x)$. 
For ample locally free sheaves $\cal V$ on complex projective manifolds,
lower bounds on $\sh(\cal V;x)$
give information about the jet separation of ``adjoint-type'' sheaves.
These are all powerful applications of Seshadri constants, with the only drawback that they require strong global positivity conditions on $\cal V$
like nefness, or even ampleness. 

In the line bundle case, on complex projective manifolds, \cite{nakamaye} introduced the moving Seshadri constant $\sh(\lVert\cal L\rVert;x)$ of $\cal L$ at $x$. It is a refinement of $\sh(\cal L;x)$, 
defined in terms of usual Seshadri constants of 
certain ample Fujita approximations of $\cal L$.
If $\cal L$ is a big and nef line bundle, then 
$\sh(\lVert\cal L\rVert;x)=\sh(\cal L;x)$.
While the definition is less intuitive, the applications are more powerful.
\cite{ELMNP} proves that the same properties mentioned in the previous
paragraph are true of $\sh(\lVert\cal L\rVert;x)$ for big line bundles $\cal L$
on complex projective manifolds.

In the forthcoming paper \cite{moving} we will extend these
to arbitrary rank.   
We will also prove a version of Theorem C for moving Seshadri constants 
that does not assume the nefness of $\cal V$.

\subsection*{Acknowledgments}
We thank Harold Blum, Yajnaseni Dutta, Lawrence Ein, S\'andor J. Kov\'acs, Yuchen
Liu, Nicholas M\textsuperscript{c}Cleerey, Mihnea Popa, Valentino Tosatti, and Yifei Zhao for useful discussions.
We are especially grateful to Krishna Hanumanthu for helpful comments on a
previous draft of this paper.
The second author would also like to thank his advisor Mircea Musta\c{t}\u{a}
for his constant support and encouragement.

\section{Background and notation}

Let $X$ be a projective scheme over an algebraically closed field.
We denote by $\operatorname{Div}(X) \otimes_{\bb Z} \bb R$ 
the space of $\bb R$-Cartier $\bb R$-divisors, 
where $\operatorname{Div}(X)$ is the group of Cartier
divisors on $X$.
\subsection{Formal twists of coherent sheaves}
We define formal twists of coherent sheaves.
See \cite[Section 6.2]{laz042} for the case of bundles.
\begin{defn}Let $\cal V$ be a coherent sheaf on $X$, and let $\lambda\in \operatorname{Div}(X) \otimes_{\bb Z} \bb R$.  
The \emph{formal twist} of $\cal V$ by $\lambda$ is the pair $(\cal V,\lambda)$,
denoted by $\cal V\langle\lambda\rangle$.
\end{defn}

When $D\in\operatorname{Div}(X)$, the formal twist $\cal V\langle D\rangle$ is the usual twist 
$\cal V\otimes\cal O_X(D)$.
Inspired by this and \cite[Example 3.2.2]{fulton84}, we can define Chern classes for formal twists by
\[
  c_i(\cal V\langle\lambda\rangle)=\sum_{j=0}^i\binom{r-j}{i-j}
  c_j(\cal V)c_1^{i-j}(\lambda),
\]
where $c_1(\lambda)$ is the image of $\lambda$ in
$\operatorname{End}(\operatorname{CH}_*(X))\otimes_{\bb Z}\bb R$,
or simply in the N\' eron--Severi space with real coefficients $N^1(X)$.

The theory of twisted sheaves has natural pullbacks. 
In particular, when $D$ is a $\bb Q$-Cartier $\bb Q$-divisor and $f\colon X'\to X$ is a finite
morphism such that $f^*D$ is actually Cartier, then $f^*\cal V\langle
f^*D\rangle$ is $f^*\cal V\otimes\cal O_{X'}(f^*D)$. The Chern classes of
twisted sheaves are natural for pullbacks.

For tensor powers and symmetric powers, we put $\cal
V\langle\lambda\rangle\otimes\cal V'\langle\lambda'\rangle\coloneqq  (\cal
V\otimes\cal V')\langle\lambda+\lambda'\rangle$ and $\Sym^n(\cal
V\langle\lambda\rangle)\coloneqq  (\Sym^n\cal V)\langle n\lambda\rangle$,
respectively. 
Generally, when we talk about extensions, subsheaves, quotients of twisted sheaves, or morphisms between twisted sheaves, 
we understand that the twist is fixed. The exception is $\Sym^*(\cal
V\langle\lambda\rangle)\coloneqq  \bigoplus_{n\geq 0}\Sym^n\cal V\langle n\lambda\rangle$.

\subsection{Positivity for twisted coherent sheaves}
Let $\bb P_X(\cal V)=\operatorname{Proj}_{\cal O_X}\bigl(\Sym^*\cal V\bigr)$ denote the space of 1-dimensional quotients of (fibers of) $\cal V$. 
Usually, we suppress $X$ from the notation.
Let $\rho\colon\bb P(\cal V)\to X$ denote the natural projection map, and let $\xi$ denote the first Chern class of the relative Serre 
$\cal O_{\bb P(\cal V)}(1)$ line bundle. 
Recall that if $D$ is Cartier on $X$, then $\bb P(\cal V)\simeq \bb P(\cal V\otimes\cal O_X(D))$ and the relative 
$\cal O(1)$ sheaves satisfy the formula 
$\cal O_{\bb P(\cal V\otimes\cal O_X(D))}(1)=\cal O_{\bb P(\cal V)}(1)\otimes\rho^*\cal O_X(D)$.
We extend these identifications formally to twists. 

\begin{defn}
Let $\cal V$ be a coherent sheaf and let $\lambda$ be an $\bb R$-Cartier $\bb R$-divisor on $X$. 
Define $\bb P(\cal V\langle\lambda\rangle)$ as $\rho\colon\bb P(\cal V)\to X$, polarized with
the $\rho$-ample $\bb R$-Cartier $\bb R$-divisor
$\cal O_{\bb P(\cal V\langle\lambda\rangle)}(1)\coloneqq  \cal O_{\bb P(\cal V)}(1)\langle\rho^*\lambda\rangle$
whose first Chern class is $\xi+\rho^*\lambda$. As above, $\xi\coloneqq   c_1(\cal O_{\bb P(\cal V)}(1))$.
\end{defn}
 
\begin{defn}
  The sheaf $\cal V$ is said to be \emph{ample} (resp.\ \emph{nef,}
  \emph{effective}) if the Cartier divisor class $\xi$ has the same property. This extends formally to twists.
\end{defn}

\begin{rmk}\label{rmk:amplesheaves}
  For locally free sheaves $\cal V$ on the projective scheme $X$, the following
  three conditions are equivalent (see \cite[Theorem 6.1.10]{laz042}):
  \begin{enumerate}
    \item[(i)] $\cal V$ is ample.
    \item[(ii)] (Global generation) For every coherent sheaf $\cal F$, the twist
      $\Sym^m\cal V\otimes\cal F$ is globally generated for $m$ sufficiently
      large.
    \item[(iii)] (Cohomological vanishing) For every coherent sheaf $\cal F$,
      the groups $H^i(X,\Sym^m\cal V\otimes\cal F)$ vanish for all $i>0$ and all
      $m$ sufficiently large.
  \end{enumerate}
  When $\cal V$ is not necessarily locally free, we still have $(\mathrm{i})
  \Leftrightarrow (\mathrm{ii})$.
  (For $\Rightarrow$, if $\cal F$ is an invertible sheaf,
then use the ampleness of $\cal O_{\bb P(\cal V)}(m)\otimes\rho^*\cal F$
for large $m$ and Lemma \ref{lem:CM}. 
Note that $\rho_*\cal O_{\bb P(\cal V)}(m)=\Sym^m\cal V$ for $m$ sufficiently large.
For an arbitrary coherent sheaf $\cal F$, it suffices to note that it can be
written as a quotient of a finite direct sum of invertible sheaves. For
$\Leftarrow$, see the proof of $(\mathrm{iv}^*)\Rightarrow(\mathrm{i})$ in
\cite[Theorem 6.1.10]{laz042}.)

$(\mathrm{i})$ also 
implies $(\mathrm{iii})$
for \emph{locally free} sheaves $\cal F$ (use the Leray spectral sequence,
the relative ampleness of $\cal O_{\bb P(\cal V)}(1)$, 
the projection formula,
and cohomology vanishing for $\cal O_{\bb P(\cal V)}(m)\otimes\rho^*\cal F$ as
in the proof for $(\mathrm{i})\Rightarrow(\mathrm{ii})$ in \cite[Theorem
6.1.10]{laz042}).
\end{rmk}
 
We can also define \emph{big} or \emph{pseudo-effective} coherent sheaves, cf.\ \cite[Definitions
5.1 and 6.1]{bundleloci}, but the definitions are more refined. See also
Definition \ref{def:vbig}.

We often see the data $\rho\colon\bb P(\cal V)\to X$ and $\xi$, even in the twisted case, as a particular case of a projective morphism $\rho\colon Y\to X$ of projective schemes with a divisor class $\xi$ on $Y$. Many times, $\xi$ will be $\rho$-nef or even $\rho$-ample, as in the case of bundles. 

\section{Definition and properties of Seshadri constants}\label{section:definition}
We start by fixing some notation for the rest of this section.
\begin{notn}\label{notn:seshnot}
Let $\rho\colon Y\to X$ be a morphism of projective schemes over an
algebraically closed field, and fix a closed point $x\in X$.
Let $\pi\colon \bl_xX\to X$ be the blow-up at $x$ with Cartier exceptional divisor $E$.
We then consider the commutative square
\[
  \xymatrix{
    Y' \ar[r]^{\pi'} \ar[d]_{\rho'} & Y \ar[d]^{\rho}\\
    \bl_xX\ar[r]_-{\pi} & X
  }
\]
where $Y'\coloneqq  \bl_{Y_x}Y$ and $Y_x\coloneqq  \rho^{-1}(x)$.
The exceptional divisor of $\pi'$ is $\rho'^*E$.
Note that the square is cartesian when $\rho$ is flat at $x$.
In any event, the $\pi'$-ampleness of $-\rho'^*E$ implies that the induced map $Y'\to Y\times_X\bl_xX$ is finite. 

\par Let $\cal C_{\rho,x}$ denote the set of irreducible curves on $Y$ that meet
$Y_x$, but are not contained in the support of $Y_x$. Let $\cal C'_{\rho,x}$ denote their strict transforms via $\pi'$.
Let $\xi$ be a numerical divisor class on $Y$.
Most of the time we assume that $\xi$ is $\rho$-nef, meaning $\xi|_{Y_t}$ is nef for all $t\in X$,
or even $\rho$-ample.
\par A case that we are particularly interested in is when $Y=\bb P_X(\cal V)$
for some coherent sheaf $\cal V$ on $X$, often locally free. In this case,
$\rho\colon\bb P(\cal V)\to X$ is the bundle map, and $\xi=c_1(\cal O_{\bb
P(\cal V)}(1))$. We denote $\cal C_{\cal V,x}\coloneqq  \cal C_{\rho,x}$ and $\cal
C'_{\cal V,x}\coloneqq  \cal C'_{\rho,x}$.
\end{notn}
\subsection{Definition and basic properties}
We begin by defining the notion of local nefness.
\begin{defn}[Local nefness] Suppose $\xi$ is $\rho$-nef. We say that $\xi$ is
\emph{nef at $x$} if $\xi\cdot C\geq 0$ for all $C\in\cal C_{\rho,x}$. 
When $Y=\bb P(\cal V)$, we also say that $\cal V$ is \emph{nef at $x$} when
the same condition holds for $\xi = c_1(\cal O_{\bb P(\cal V)}(1))$.
\end{defn}
 
\begin{ex}\label{ex:ggnef} 
If a coherent sheaf $\cal V$ is globally generated at $x$, i.e.,
$H^0(X,\cal V)\otimes\cal O_X\to\cal V$ is surjective at $x$,
then $\cal V$ is nef at $x$. (Since $\rho^*\cal V\to\cal O_{\bb P(\cal V)}(1)$
is surjective, we find that $\cal O_{\bb P(\cal V)}(1)$ is globally
generated along the fiber $\rho^{-1}x=\bb P(\cal V(x))$. 
If $C$ is a curve that meets $\bb P(\cal V(x))$ without being contained in it,
and if $y\in C\cap\bb P(\cal V(x))$, then we can find an effective
representative of $\xi$ that does not pass 
through $y$, hence it does not contain $C$.
It follows that $\xi\cdot C\geq 0$.)
\qed
\end{ex}

\begin{rmk}\label{rmk:detectnef}
If $\xi$ is $\rho$-nef, then $\xi$ is nef on $Y$ if and only if $\xi$ is nef at all $x\in X$. (One direction is clear. 
The other is immediate from the $\rho$-nefness of $\xi$.)
\end{rmk}

We now define the following measure of local nefness at $x$.
We believe these constants were first defined explicitly for ample locally free
sheaves by Hacon \cite[p.\ 769]{hacon}, although they appear implicitly in the
work of Beltrametti, Schneider, and Sommese \cite{BSS93,BSS96}.

\begin{defn}The \emph{Seshadri constant} of $\xi$ at $x$ is
\[
  \sh(\xi;x)\coloneqq  \inf_{C\in\cal C_{\rho,x}}\left\{\frac{\xi\cdot C}{\mult_x\rho_*C}\right\}.
\]
When $Y=\bb P(\cal V)$, put $\sh(\cal V;x)\coloneqq  \sh(\cal O_{\bb P(\cal V)}(1)\;x)$.
When $\cal C_{\rho,x}$ is empty, set $\sh(\xi;x)=\infty$.
\end{defn}

\begin{rmk}The Seshadri constant descends to a well-defined function
$\sh(-;x)\colon N^1(Y) \to \bb R$
that is homogeneous and concave, i.e., $\sh((1-t)\xi+t\xi';x)\geq(1-t)\sh(\xi;x)+t\sh(\xi';x)$ for all $t\in[0,1]$.
\end{rmk}

\begin{rmk}[Multipoint version] If $x_1,\ldots,x_m$ are finitely many points in $X$, one can similarly define
$\sh(\xi;\{x_1,\ldots,x_m\})=\inf\left\{\frac{\xi\cdot C}{\mult_{x_1}\rho_*C+\ldots+\mult_{x_m}\rho_*C}\right\}$,
where $C$ ranges through curves on $Y$ with $\mult_{x_1}\rho_*C+\ldots+\mult_{x_m}\rho_*C\neq 0$.
\end{rmk}

\begin{prop}\label{prop:altintsh}
  If $\xi$ is $\rho$-nef, then
\[
  \sh(\xi;x)=\sup\bigl\{t\st (\pi'^*\xi-t\rho'^*E)\cdot C'\geq 0\text{ for all
  }C'\in\cal C'_{\rho,x}\bigr\}.
\]
\end{prop}

\noindent Note that the curves in $\cal C_{\rho,x}$ are precisely the irreducible curves $C$ on $Y$ for which $\mult_x\rho_*C>0$.
See also \cite{hacon} for the case of bundles. 
\begin{proof}Let $C'$ be the strict transform of $C$ on $Y'$ via $\pi'$.
We then have
\[
  \mult_x\rho_*C=E\cdot\rho'_*C'=\rho'^*E\cdot C',
\]
hence $(\pi'^*\xi-t\rho'^*E)\cdot C'\geq 0$
if and only if $\frac{\xi\cdot C}{\mult_x\rho_*C}\geq t$. 
\end{proof}

\begin{ex}When $\rho$ is the identity morphism $X \to X$ and $\xi$ is nef, then
$\sh(\xi;x)$ is the classical Seshadri constant of the divisor class $\xi$ at
$x$; see \cite[Proposition 5.1.5]{laz04}.
\end{ex}

\begin{ex}When $\rho=\pi$ is the blow-up of $x$ and $\xi=-E$, then $\sh(\xi;x)=-1$. 
In fact for all curves $C$ on $\bl_xX$ that meet $E$, without being contained in it, we have $\frac{-E\cdot C}{\mult_x\pi_*C}=-1$.
\end{ex}

\begin{rmk}\label{rmk:easyseshadri} Assume that $\xi$ is $\rho$-nef. We have the following:
\begin{enumerate}[(a)]
\item $\sh(\xi;x)\geq 0$ if and only if $\xi$ is nef at $x$.
\item If $C'$ is an irreducible curve on $Y'$ that is contained in the exceptional locus $\rho'^{-1}E$ of $\pi'$,
then
\[
  (\pi'^*\xi-t\rho'^*E)\cdot C'=\xi\cdot\pi'_*C'-tE\cdot\rho'_*C'\geq 0 
\]
for all $t\geq 0$. The inequality is strict if $t>0$ and $C'$ is not contracted by $\rho'$, or if $C'$ is not contracted by $\pi'$ and $\xi$ is $\rho$-ample. (Use that $\xi$ is nef on $Y_x$, and that $-E$ is ample on $E$.)
\item If $\xi$ is nef, then $\sh(\xi;x)=\sup\bigl\{t \in \bb R_{\ge 0}\st
  \pi'^*\xi-t\rho'^*E\in\Nef^1(Y')\bigr\}$. 
In particular, if $\cal V$ is nef, then 
\[
  \sh(\cal V;x)=\sup\bigl\{t \in \bb R_{\ge 0}\st \pi^*\cal V\langle-tE\rangle\text{ is nef}\bigr\}.
\]
(The twisted bundle $\pi^*\cal V\langle-tE\rangle$ is nef if and only if
$\pi'^*\xi-t\rho'^*E$ is nef on $\bb P(\pi^*\cal V)$. The irreducible curves on
$\bb P(\pi^*\cal V)$ are either in $\cal C'_{\cal V,x}$, are in the exceptional
locus of $\pi'$, or do not intersect the support of $\rho'^*E$. From part (b),
and using the nefness of $\pi'^*\xi$, we find that the nefness of
$\pi'^*\xi-t\rho'^*E$ can be verified on the curves in $\cal C'_{\cal V,x}$. The
case for general $Y$ and $\xi$ is analogous.)
\end{enumerate}
\end{rmk}

As is the case for divisors \cite[Theorem 1.4.13]{laz04}, Seshadri
constants can detect whether $\xi$ is ample.
\begin{thrm}[Seshadri ampleness criterion]\label{seshadriample}
If $\xi$ is $\rho$-ample, then
$\xi$ is ample if and only if
\begin{equation}\label{eq:seshadriampleinf}
  \inf_{x\in X}\sh(\xi;x)>0.
\end{equation}
In particular, if $\cal V$ is a locally free sheaf on $X$, then
$\cal V$ is ample if and only if $\inf_{x\in X}\sh(\cal V;x)>0.$
\end{thrm}

\noindent See also \cite[Example 6.1.20]{laz04}.

\begin{proof}
Assume that the infimum in \eqref{eq:seshadriampleinf} is positive, but that $\xi$ is not ample.
In any case, $\xi$ is nef by Remarks \ref{rmk:detectnef} and
\ref{rmk:easyseshadri}.(a).
By the Seshadri ampleness criterion for divisors \cite[Theorem 1.4.13]{laz04}, $\inf_{y\in Y}\sh(\xi;y)=0$.
Hence there exist closed points $y_m\in Y$ and irreducible curves $C_m$ through $y_m$ with 
\[
  \xi\cdot C_m<\frac{1}{m}\mult_{y_m}C_m.
\]
We claim that $C_m$ is not contracted by $\rho$ for infinitely many $m$.
Indeed, suppose that the curves $C_m$ are contracted by $\rho$ for all $m$, in which case
\begin{equation}\label{eq:seshadriampletocont}
  \adjustlimits\inf_{x\in X}\inf_{y\in Y_x}\sh\bigl(\xi\rvert_{Y_x};y\bigr)=0.
\end{equation}
Let $h$ be a sufficiently ample divisor class on $X$ such that $\xi+\rho^*h$ is ample.
Then, $\inf_{y\in Y}\sh(\xi+\rho^*h;y)>0$, and
in particular,
\[
  \adjustlimits\inf_{x\in X}\inf_{y\in Y_x}\sh\bigl((\xi+\rho^*h)\rvert_{Y_x};y\bigr)>0.
\]
But $(\xi+\rho^*h)\rvert_{Y_x}=\xi\rvert_{Y_x}$, contradicting
\eqref{eq:seshadriampletocont}.
This shows the claim.
\par From the claim, $\rho\rvert_{C_m}$ is finite for all sufficiently large $m$.
Writing $x_m\coloneqq  \rho(y_m)$, the inequality $\mult_{x_m}\rho_*C_m\geq
\mult_{y_m}C_m$ (see \cite[Lemma 2.3]{ful17}) leads to a contradiction.

Conversely, assume that $\xi$ is ample. Let $h$ be ample on $X$. Then,
$\xi-\epsilon \rho^*h$ is ample for sufficiently small $\epsilon>0$, and
for all $C\in\cal C_{\rho,x}$, since $\rho\rvert_C$ is finite, 
\[
  \frac{\xi\cdot C}{\mult_x\rho_*C}=\frac{(\xi-\epsilon\rho^*h)\cdot C}{\mult_x\rho_*C}+\frac{\epsilon\rho^*h\cdot C}{\mult_x\rho_*C}\geq
\epsilon\frac{h\cdot\rho_*C}{\mult_x\rho_*C}\geq\epsilon\cdot\sh(h;x).
\]
Taking the infimum over all $x \in X$, we see that $\xi$ is ample by the
classical Seshadri ampleness criterion for divisors \cite[Theorem
1.4.13]{laz04}.
\end{proof}

\begin{rmk}\label{rmk:compareless1}In the case of sheaves, the first part of the
previous proof can be adapted to show the following: \emph{If there exists $y
\in \bb P(\cal V(x))$ such that $0 \le \sh(\xi;y) < 1$, then $\sh(\cal V;x) \le
\sh(\xi;y)$.}
For arbitrary $\rho$ and $\rho$-ample $\xi$, a similar statement holds with $1$
replaced by $\inf_{y\in Y_x}\sh(\xi\rvert_{Y_x};y)$, which is in any case strictly positive.
(The inequality $\sh(\xi;y)<1$ proves that the Seshadri constant of $\xi$ at $y$ is not approximated by intersecting with curves in $\bb P(\cal V(x))$, 
since $\sh(\xi\rvert_{\bb P(\cal V(x))};y)=1$. For curves in $\cal C_{\cal V,x}$ that
pass through $y$, use the inequality $\mult_y C\leq\mult_x\rho_*C$ from
\cite[Lemma 2.3]{ful17}.)
\par Furthermore, for arbitrary $\rho$ and $\rho$-ample
$\xi$, we have the following: \emph{If there exists $y \in Y_x$ such that
$\sh(\xi;y) < 0$, then $\sh(\xi;x) < 0$.}
(If $\sh(\xi;y) < 0$, then there exists $C\in\cal C_{\rho,x}$ through $y$ with
$\xi\cdot C<0$.)
\qed
\end{rmk}

One can also characterize Seshadri constants in terms of all varieties
intersecting $Y_x$, instead of just curves.
\begin{prop}\label{prop:seshadrihigherdimension}
If $\xi$ is nef, then 
\begin{equation}\label{eq:uglyomg}\sh(\xi;x)\leq\left(\frac{\xi^{\dim
W}\cdot[W]}{{\binom{\dim W}{\dim\rho(W)}}\cdot\mult_x\rho(W)\cdot(\xi^{\dim W_{x'}}[W_{x'}])}\right)^{1/\dim \rho(W)},\end{equation}
as $W$ ranges through the subvarieties of $Y$ that meet $Y_x$ without being contained in it. 
In the above, $W_{x'}$ is a fiber over the flat locus of $W\to\rho(W)$.

\end{prop}
\noindent If $X$ is a variety and $Y=\bb P(\cal V)$ for a locally free sheaf $\cal V$ of rank $r$, then in particular by considering $W=Y$, we obtain
\begin{equation}\sh(\cal V;x)\leq\sqrt[n]{\frac{s_n(\cal V^{\vee})}{{\binom{n+r-1}{n}}\cdot\mult_xX}},
\end{equation}
where $s_n(\cal V^{\vee})=(\xi^{n+r-1})$ is the $n$-th Segre class of $\cal V^{\vee}$ (see \cite[\S3.1]{fulton84}\footnote{Duality is present because \cite{fulton84} uses projective bundles of lines
instead of quotients.}). 
This is a generalization of the rank one case $\sh(\cal L;x)\leq\sqrt[n]{\frac{(\cal L^n)}{\mult_xX}}$ in \cite[Proposition 5.1.9]{laz04}.
A transcendental generalization is \cite[Theorem 4.6]{tosatti}.

\begin{ex}
  Put $n\coloneqq  \dim X$ and assume that $\cal V$ is locally free of rank $r$ and nef. When considering $W=\rho^{-1}Z\subseteq\bb P(\cal V)$ for some subvariety $Z\subseteq X$ of codimension $i$, we obtain
\[
\sh(\cal V;x)\leq\left(\frac{\xi^{n-i+r-1}\cdot[\rho^*Z]}{{\binom{n-i+r-1}{n-i}}\cdot\mult_xZ}\right)^{\frac 1{n-i}}=
\left(\frac{s_{n-i}(\cal V^{\vee})\cap[Z]}{{\binom{n-i+r-1}{n-i}}\cdot\mult_xZ}\right)^{\frac 1{n-i}},
\]
where $s_{n-i}(\cal V^{\vee})\cap[Z]=\xi^{n-i+r-1}\cdot[\rho^*Z]$ is the evaluation of the Segre class of degree 
$n-i$ of $\cal V^{\vee}$ on the fundamental class of $Z$ (see \cite[\S 3.1]{fulton84}).
  These bounds are similar to the ones appearing in \cite[Theorem 1.5.a]{hacon}.

  We thank Valentino Tosatti for suggesting this example.
\end{ex}

\begin{rmk}[Relation with other Seshadri constants] With hypotheses
as in the previous example, taking the infimum over all $Z$ of fixed codimension $i$, we obtain
\[
\sh(\cal V;x)\leq\left(\frac 1{\binom{n-i+r-1}{n-i}}\cdot\sh\bigl(s_{n-i}(\cal V^{\vee});x\bigr)\right)^{\frac 1{n-i}},
\]
where the Seshadri constant of the nef dual class $s_{n-i}(\cal V^{\vee})$
on the right is defined as in \cite[\S8]{ful17}.

We thank Nicholas M\textsuperscript{c}Cleerey for suggesting this example.
\end{rmk}

Formula \eqref{eq:uglyomg} looks more familiar when $W=C$ is a curve in $\cal C_{\rho,x}$. Note that $\mult_x\rho_*[C]=\mult_x\rho(C)\cdot[C_{x'}]$, since $\deg[C_{x'}]=\deg(\rho|_C)$.

\begin{proof}[Proof of Proposition \ref{prop:seshadrihigherdimension}]
Let $W$ be as above, and let $W'$ be its strict transform in $Y'$. By Remark \ref{rmk:easyseshadri}.(c) we have 
$(\pi'^*\xi-\sh(\xi;x)\rho'^*E)^{\dim W'}\cdot[W']\geq 0$.
By restricting to $W'$ we can assume without loss of generality that $W'=Y'$,
that $\rho$ is surjective, and that $X$ is a variety. 
Let $n\coloneqq\dim X$ and $e\coloneqq  \dim Y-n$, with $e\geq 0$. 
We have 
\begin{align*}0\leq \bigl(\pi'^*\xi-\sh(\xi;x)\rho'^*E\bigr)^{n+e} & =
  \sum_{k=0}^n{\binom{n+e}{k}}\bigl(-\sh(\xi;x)\rho'^*E\bigr)^k\pi'^*\xi^{n+e-k}\\
  & \leq \xi^{n+e}+{\binom{n+e}{n}}\bigl(-\sh(\xi;x)\rho'^*E\bigr)^n\cdot\pi'^*\xi^{e}\\
    & = \xi^{n+e}-{\binom{n+e}{n}}\cdot\mult_xX\cdot\sh^n(\xi;x)\cdot(\xi^e\cdot[Y_{x'}]).
\end{align*}
The first equality holds since $(E^k)=0$ for $k>n$.
The second inequality is a consequence of the projection formula for $\pi'$.
Pushing forward $-(-\rho'^*E)^k$ produces a pseudo-effective class, since $-E|_E$ is ample. 
In the last equality, we used that $(-E)^n=-\mult_xX$.
This implies $\pi'_*(\rho'^*(-E)^n)=-\mult_xX\cdot F$, where $F$ is a fiber over the flat locus of $\rho|_W:W\to\rho(W)$. 
\end{proof}

\begin{rmk}
With hypotheses as in the proposition, assume that $\xi$ is ample.
We show that there exists a subvariety $W'\subseteq Y'$, which is the strict 
transform of some $W\subseteq Y$ that meets $Y_x$ 
without being contained in it, such that
\[(\pi'^*\xi-\sh(\xi;x)\rho'^*E)^{\dim W'}\cdot[W']=0.\]

For this, let $W'\subset Y'$ be a subvariety that observes the failure of ampleness of $\pi'^*\xi-\sh(\xi;x)\rho'^*E$, i.e., 
$(\pi'^*\xi-\sh(\xi;x)\rho'^*E)^{\dim W'}\cdot[W']=0$.
These exist by \cite{cp90,birkar} over arbitrary fields for nef $\bb R$-Cartier $\bb R$-divisors, 
extending the Nakai--Moishezon criterion for nef Cartier divisors.
We want to show that $W\coloneqq  \pi'(W')$ meets $Y_x$ without being contained in it.

If $W'$ does not meet $\rho'^{-1}E$, then $(\pi'^*\xi-\sh(\xi;x)\rho'^*E)^{\dim W'}\cdot[W']=\xi^{\dim W}\cdot[W]>0$ because $\xi$ is ample.
This is a contradiction, therefore $W$ meets $Y_x$. 
If $W'$ is contained in the exceptional locus of $\pi'$, then $W$ is contained in $Y_x$.
Using that $Y'\to Y\times_X\bl_xX$ is finite, we deduce that $W'\to W\times E$ is also finite.
Using that $-E|_E$ is ample, it follows that $(\pi'^*\xi+t(\rho'^*(-E)))|_{W'}$ is ample for all $t>0$. 
Since $\xi$ is ample, in any case $\sh(\xi;x)>0$.
We again obtain a contradiction, hence $W$ is not contained in $Y_x$.\qed
\end{rmk}

\begin{rmk}When $\cal V$ is ample on $X$, it is tempting
to believe that $\sh(\cal V;x)$ should be controlled by subvarieties of $X$
through $x$. In other words, one would expect that equality in the
proposition is achieved by some $W=\rho^{-1}Z$ for $Z$ a 
subvariety of $X$ containing $x$. However, this is not true.

As in \cite[p.~771]{hacon}, consider $X=\bb P^1$ and $\cal V=\cal O_X(1)\oplus\cal O_X(2)$. 
In this case $E=Y_x$. From Remark \ref{rmk:easyseshadri}.(c),
we deduce $\sh(\cal V;x)=1$. The only subvariety of $\bb P(\cal V)$
that achieves equality in \eqref{eq:uglyomg} is $W=\bb P(\cal O_X(1))$,
embedded via the quotient $\cal O_X(1)\oplus\cal O_X(2)\twoheadrightarrow\cal O_X(1)$.
\end{rmk}

\subsection{Functoriality I}
We now discuss how Seshadri constants behave under various operations.
\begin{lem}[Quotients]\label{lem:shquot}
Assume that $\xi$ is $\rho$-nef.
Let $\imath\colon Z\to Y$ be a morphism of projective schemes. Then,
\begin{equation}\label{eq:shquotineq}
  \sh(\imath^*\xi;x)\geq\sh(\xi;x),
\end{equation}
and equality holds if $\imath$ is surjective.
In particular, if $\cal V\to\cal Q$ is a surjective morphism of coherent sheaves
on $X$, then $\sh(\cal Q;x)\geq\sh(\cal V;x)$.
\end{lem}
\begin{proof}
Let $C\in\cal C_{\rho\circ\imath,x}$, and write $C'\coloneqq  \imath(C)\in\cal C_{\rho,x}$.  
We have $\imath_*C=dC'$ for some $d\geq 1$. By the projection formula, we have
\[
  \frac{\imath^*\xi\cdot C}{\mult_x(\rho\circ\imath)_*C}=\frac{\xi\cdot
  dC'}{\mult_x\rho_*(dC')}=\frac{\xi\cdot C'}{\mult_x\rho_*C'}.
\]
Taking the infimum over all $C \in C_{\rho\circ\imath,x}$, since $\cal C_{\rho,x}$ may contain curves that are not of form $C'$ as above,
we deduce $\sh(\imath^*\xi;x)\geq\sh(\xi;x)$. When $\imath$ is surjective, every
curve in $\cal C_{\rho,x}$ is of form $C'$ as above, hence equality holds in
\eqref{eq:shquotineq}.

For the last statement, note that there is a closed immersion $\bb P_X(\cal Q)\hookrightarrow\bb P_X(\cal V)$ such that the restriction of $\cal O_{\bb P(\cal V)}(1)$ is $\cal O_{\bb P(\cal Q)}(1)$.
\end{proof}

\begin{lem}[Generically finite pullbacks]\label{lem:pullbacks} 
Consider a \emph{cartesian} diagram of projective schemes
\[
  \xymatrix{
    Y'\ar[r]^{f'}\ar[d]_{\rho'} \ar@{}[dr]|*={\square} & Y\ar[d]^{\rho}\\
    X'\ar[r]_{f}& X
  }
\]
Let $x'\in X'$ be a closed point in the finite locus of $f$. Put $x=f(x')$.
Let $\xi$ be $\rho$-ample on $Y$. 
If $\xi$ is nef at $x$, then
\begin{align*}
  \sh(f'^*\xi;x') &\geq \sh(\xi;x).
\intertext{When $\sh(\xi;x)<0$, and $f$ is surjective, we have $\sh(f'^*\xi;x')\leq \sh(\xi;x)$.
In particular, if $f:X'\to X$ is a generically finite morphism of projective
varieties, and if
$\cal V$ be a locally free sheaf on $X$, then}
  \sh(f^*\cal V;x') &\geq \sh(\cal V;x)
\end{align*}
for all $x'\in X'$ such that $f$ is finite around $x'$ and $\sh(\cal V;x)\geq 0$. When $\sh(\cal V;x)<0$, and $f$ is surjective, we have $\sh(f^*\cal V;x')\leq\sh(\cal V;x)$.
\end{lem}

\begin{proof}Assume first $\sh(\xi;x)\geq 0$. Let $C\in \cal C_{\rho',x'}$. Since $f$ is finite around $x'$, we deduce $f'(C)\in \cal C_{\rho,x}$.
Let $d\geq 1$ be defined by $f'_*C=d\cdot f'(C)$.
We then have
\begin{align*}
  \frac{f'^*\xi\cdot C}{\mult_x\rho'_*C}=
  d\cdot\frac{\xi\cdot f'(C)}{\mult_x\rho'_*C}
  \MoveEqLeft[3]\geq d\cdot\frac{\xi\cdot f'(C)}{\mult_{f(x)}f_*\rho'_*C}
  =d\cdot\frac{\xi\cdot f'(C)}{\mult_{f(x)}\rho_*f'_*C}\\
  &=\frac{\xi\cdot f'(C)}{\mult_{f(x)}\rho_*f'(C)}\geq\sh(\xi;x).
\end{align*}
The first inequality says that multiplicity increases under finite pushforwards.
See for example \cite[Lemma 2.3]{ful17}. We conclude by taking the infimum over all $C\in\cal C_{\rho',x'}$.

When $\sh(\xi;x)<0$ and $f$ is surjective, for all sufficiently small
$\delta>0$, let $C_{\delta}\in\cal C_{\rho',x'}$ such
that
\[
  \frac{\xi\cdot f'(C_{\delta})}{\mult_{f(x)}\rho_*f'(C_{\delta})}<\sh(\xi;x)+\delta.
\]
From $\xi\cdot f'(C_{\delta})<0$ it follows that $C_{\delta}$ is not contracted by $\rho'$, and as in the previous case,
\[
  \frac{f'^*\xi\cdot C_{\delta}}{\mult_x\rho'_*C_{\delta}}<\sh(\xi;x)+\delta.
  \qedhere
\]
\end{proof}

\begin{lem}[Box Products] For $i\in\{1,2\}$, let $\rho_i\colon Y_i\to X$ be morphisms of projective schemes, and let $\xi_i$ be a $\rho_i$-ample divisor on $Y_i$.
Fix $x\in X$. Let $\rho\colon Y_1\times_XY_2\to X$ be the induced morphism.
Denote by $p_i\colon Y_1\times_XY_2\to Y_i$ the two projections, 
and set $\xi_1\boxtimes\xi_2\coloneqq   p_1^*\xi_1+p_2^*\xi_2$. 
Then 
\[
  \sh(\xi_1\boxtimes\xi_2;x)\geq\sh(\xi_1;x)+\sh(\xi_2;x).
\]
Equality holds for equal input data $(Y_1,\rho_1,\xi_1)=(Y_2,\rho_2,\xi_2)$.
Analogous statements hold for products of finitely many $\rho_i$.
\end{lem}

\begin{proof}Let $C\in\cal C_{\rho,x}$. Then $p_i(C)\in\cal C_{\rho_i,x}$. 
From the projection formula,
\[\frac{\xi_1\boxtimes\xi_2\cdot C}{\mult_x\rho_*C}=\frac{\xi_1\cdot p_{1*}C}{\mult_x\rho_{1*}p_{1*}C}+\frac{\xi_2\cdot p_{2*}C}{\mult_x\rho_{2*}p_{2*}C}
\geq\sh(\xi_1;x)+\sh(\xi_2;x).\]
For equal input data $Y=Y_i$, $\rho_i$, and $\xi_i$, the Seshadri constants 
$\sh(\xi_i;x)$ on $Y_i=Y$ are 
approximated by the same curves $C$ on $Y$. 
Apply the formula above to the diagonal curve $\Delta_C\subset C\times_XC\subset Y\times_XY$.
\end{proof}

\subsection{Restrictions to curves}
Our goal in this subsection is to describe how Seshadri constants on a scheme
$X$ can be characterized by their behavior on curves in $X$.
This example is fundamental to the development of the theory, since it allows us
to reduce to the case where $X$ is a smooth projective curve.
\begin{rmk}\label{rmk:curves}Assume that $X$ is a projective curve and that $\xi$ is $\rho$-nef. 
  Then,
  \[
    \sh(\xi;x)=\frac 1{\mult_xX}\cdot\sup\{t\st \xi-tf\in\Nef^1(Y)\},
  \]
  where $f$ is the class of a general fiber of $\rho$. In particular, $\mult_xX\cdot \sh(\xi;x)$ is independent of $x$ in this case.
If $\cal V$ is a coherent sheaf on $X$, then 
\[
  \sh(\cal V;x)=\frac 1{\mult_xX}\cdot\sup\{t\st \cal V\langle-tx_0\rangle\ \text{is nef}\},
\]
where $x_0$ denotes a $\bb Q$-Cartier $\bb Q$-class of degree 1. 
(Since $X$ is a curve, 
$\deg E=\mult_xX$.
The set $\cal C_{\rho,x}$ is the set of curves in $Y$ that dominate $X$, 
hence it is independent of $x$.
Using the $\rho$-nefness of $\xi$, 
it follows that $(\pi'^*\xi-t\rho'^*E)\cdot C'\geq 0$ for all 
$C'\in\cal C'_{\rho,x}$ if and only if $\pi'^*\xi-t\rho'^*E$ is nef.
However, $\pi'^*\xi-t\rho'^*E=\pi'^*(\xi-(\mult_xX)tf)$ is nef on $Y'$ if and only if $\xi-(\mult_xX)tf$ is nef on $Y$.)
\end{rmk}

We can now give the following generalization of \cite[Theorem 3.1]{hacon}.

\begin{ex}[Curves]\label{ex:curves} If $X$ is a (possibly singular) integral projective
  curve over an algebraically closed field $k$ and $\nu\colon X'\to X$ denotes
  the normalization, and
if $\cal V$ is a coherent sheaf on $X$, then
\begin{equation}\label{eq:curvechar}
  \sh(\cal V;x)=\frac{\overline{\mu}_{\rm min}(\nu^*\cal V)}{\mult_xX}.
\end{equation}
For the purpose of explaining notation, assume that $X$ is a smooth projective curve. 
The slope of a bundle $\cal V$ on $X$ is 
\[\mu(\cal V)\coloneqq  \frac{\deg\cal V}{{\rm rank}\,\cal V}.\]
By convention, the slope of torsion sheaves is infinite.
The smallest slope of any quotient (of positive rank) of $\cal V$ is denoted by $\mu_{\rm min}(\cal V)$. 
A quotient of $\cal V$ with minimal slope exists, and is determined by the Harder--Narasimhan filtration of $\cal V$.
In characteristic $0$, set $\overline{\mu}_{\rm min}(\cal V)\coloneqq  \mu_{\rm min}(\cal V)$.
In characteristic $p>0$, let $F\colon X\to X$ be the absolute Frobenius morphism, and consider
\[ \overline{\mu}_{\mathrm{min}}(\cal V) \coloneqq   \lim_{n \to \infty}
\frac{\mu_{\mathrm{min}}\bigl( (F^n)^*\cal V\bigr)}{p^n}.\]
The sequence is weakly decreasing and eventually stationary. 
In fact, \cite[Theorem 2.7]{Langer04} proves that there exists 
$\delta=\delta_{\cal V}\geq 0$ such that the Harder--Narasimhan filtration of $(F^{\delta+n})^*\cal V$ is 
the pullback of the Harder--Narasimhan filtration of $(F^{\delta})^*\cal
V$.\footnote{\cite{Langer04} uses the notation $L_{\rm min}(\cal V)$ for $\overline{\mu}_{\rm min}(\cal V)$.}
In particular, $\overline{\mu}_{\rm min}(\cal V)=\frac{\mu_{\rm min}((F^{\delta})^*\cal V)}{p^{\delta}}$ 
is the smallest normalized slope of any quotient of any iterated Frobenius pullback $(F^n)^*\cal V$.

Note that torsion is irrelevant when computing $\mu_{\rm min}$ or $\overline\mu_{\rm min}$. It only affects the slope $\mu(\cal V)\geq\mu(\cal V_{\rm tf})$. 
\vskip.25cm
\noindent(For the proof of \eqref{eq:curvechar}, assume first that $X$ is smooth. 
From Remark \ref{rmk:curves}, the Seshadri constant is independent of $x\in X$, and verifies the linearity 
$\sh(\cal V\langle\lambda\rangle;x)=\sh(\cal V;x)+\deg\lambda$. 
Furthermore, slopes respect the same formula 
$\mu(\cal V\langle\lambda\rangle)=\mu(\cal V)+\deg\lambda$, 
and similarly for $\mu_{\rm min}$ and $\overline{\mu}_{\rm min}$. 

In characteristic zero, we are then free to assume that $\mu_{\rm min}(\cal V)=0$.
Hartshorne's Theorem \cite[Theorem 6.4.15]{laz04} (which is only valid in characteristic zero; see \cite[Example 3.2]{har71})
shows that $\cal V$ is nef. In particular, $\sh(\cal V;x)\geq 0$ for all $x\in X$.
By the assumption $\mu_{\rm min}(\cal V)=0$,
there exists a quotient map $\cal V\twoheadrightarrow\cal Q$ 
with $\cal Q$ nonzero, nef, semistable, and $\mu(\cal Q)=0$.
Since $\sh(\cal V;x)\leq\sh(\cal Q;x)$ by Lemma \ref{lem:shquot}, it is then
enough to treat the case when $\cal V=\cal Q$ is nef of degree 0.
In this case, one can use Remark \ref{rmk:easyseshadri}(c), where the blow-up $\pi$ of $x\in X$ is the identity, 
and the ``exceptional'' divisor $E$ is $\cal O_X(x)$.

In positive characteristic, the proof is analogous after replacing $\mu_{\rm min}$ with $\overline{\mu}_{\rm min}$, 
in view of \cite[Theorem 1.1]{BP14}, which proves that $\cal V$ is nef iff $\overline{\mu}_{\rm min}(\cal V)\geq 0$.\footnote{\cite{BP14} uses the notation $\theta_{\cal V,1}$ for our $\overline{\mu}_{\rm min}(\cal V)$.}
The result was seemingly first proved by Barton \cite[Theorem 2.1]{Barton71}, 
and stated explicitly by Brenner in \cite[Theorem 2.3]{Brenner04} and \cite[p.\ 534]{Brenner06}, 
Biswas in \cite[Theorem 1.1]{Biswas05}, and Zhao in \cite[Theorem 4.3]{Zhao17}.

When $X$ is singular, then from the projection formula, one finds
$\sh(\cal V;x)=\frac{\sh(\nu^*\cal V)}{\mult_xX}$, where $\sh(\nu^*\cal V)$ is the Seshadri constant of $\nu^*\cal V$ at any point of $X'$.
)\qed
\end{ex}

\begin{cor}[Seshadri constants for sheaves via restrictions to curves]\label{cor:shviacurves} Let $X$ be a projective scheme of arbitrary dimension over an algebraically closed field. Fix $x\in X$ a closed point and $\cal V$ a coherent (twisted) sheaf on $X$.
Then 
\[\sh(\cal V;x)=\inf_{x\in C\subset X}\frac{\overline{\mu}_{\rm min}(\nu^*\cal
V)}{\mult_xC},\]
where $C$ ranges through the set of irreducible curves through $x$ on $X$, where
$\nu\colon C'\to C$ is the normalization, 
and $\overline{\mu}_{\min}$ is defined as above.
\end{cor}

\noindent Note that torsion subsheaves whose supports have positive dimension may influence the result.

\begin{proof}Use $\cal C_{\cal V,x}=\bigcup_{x\in C\subset X}\cal C_{\cal V|_C,x}$ to deduce 
that $\sh(\cal V;x)=\inf_{x\in C\subset X}\sh(\cal V|_C;x)$.
The result then follows from Example \ref{ex:curves}.
\end{proof}

\subsection{Functoriality II}
Using our description of Seshadri constants via restrictions to curves
in Corollary \ref{cor:shviacurves}, we can describe
the behavior of Seshadri constants under more operations.

\begin{lem}[Homogeneity]\label{lem:homogeneous}
Assume that $\cal V$ is a coherent (twisted) sheaf on a projective scheme $X$. Fix $x\in X$. Then
\[\sh(\Sym^d\cal V;x)=d\cdot\sh(\cal V;x).\]
\end{lem}
\begin{proof}By Corollary \ref{cor:shviacurves} and since symmetric powers are
compatible with pullbacks, it is enough to consider the case of curves.
By normalizing, we may assume that $X$ is a smooth projective curve.
After iterated Frobenius pullback, we may assume that $\overline{\mu}_{\rm min}=\mu_{\rm min}$ throughout.
Note that slopes respect the formula $\mu(\Sym^d\cal V)=d\cdot\mu(\cal V)$ for locally free sheaves $\cal V$.
From any quotient $\cal V\twoheadrightarrow\cal Q$ of slope $\mu(\cal Q)$ we obtain the quotient $\Sym^d\cal V\twoheadrightarrow\Sym^d\cal Q$ of slope $d\cdot\mu(\cal Q)$.
This proves the ``$\leq$'' inequality by Lemma \ref{lem:shquot}.

For the inequality ``$\geq$'', we note that $\cal V\langle-{\mu}_{\rm min}(\cal V)\rangle$ is nef and not ample (cf.\ \cite[Theorem 1.1]{BP14}).
Thus, \cite[Theorem 6.2.12(iii)]{laz042} implies so is 
\[
  \Sym^d\bigl(\cal V\langle-\mu_{\rm min}(\cal V)\rangle\bigr)=\bigl(\Sym^d\cal
  V\bigr)\langle- d\cdot\mu_{\rm min}(\cal V)\rangle.
\]
In particular, the latter can have no (twisted) quotients of negative slope,
proving ``$\geq$''.
\end{proof}

\begin{cor}Let $\cal V$ be a (twisted) locally free 
sheaf of finite rank on the projective variety $X$.
Let $\nu_d:\bb P(\cal V)\to\bb P(\Sym^d\cal V)$ denote the
relative Veronese embedding. Then  
\[\nu_d^*\Nef^1\bigl(\bb P(\Sym^d\cal V)\bigr)=\Nef^1\bigl(\bb P(\cal
V)\bigr)\qquad\mbox{and}\qquad \nu_{d*}\Eff_1\bigl(\bb P(\cal V)\bigr)=\Eff_1\bigl(\bb P(\Sym^d\cal
V)\bigr).\]
\end{cor}

\begin{proof}The second equality follows from the first by duality. 
Let $\rho_d:\bb P(\Sym^d\cal V)\to X$ be the bundle map with relative
Serre bundle $\xi_d$ such that $\nu_d^*\xi_d=d\xi$. 
Since $\cal V$ is locally free, the N\' eron--Severi
spaces of $\bb P(\cal V)$ and $\bb P(\Sym^d\cal V)$ are generated
by the pullbacks of $N^1(X)$, and by $\xi$ and $\xi_d$ respectively. 
If $\delta\in N^1(X)$, it is enough to prove that
$d(\xi-\rho^*\delta)$ is nef if and only if $\xi_d-d\rho_d^*\delta$ is nef.
In other words, that $\cal V\langle-\delta\rangle$ is nef if and only if
$(\Sym^d\cal V)\langle-d\delta\rangle=\Sym^d(\cal V\langle-\delta\rangle)$
is nef. 
This is immediate from Lemma \ref{lem:homogeneous}
and from Remark \ref{rmk:easyseshadri}.(a).
\end{proof}

\begin{rmk}
With notation as in the corollary, when the characteristic of the base field is zero, then 
$\nu_d^*\Eff^1\bigl(\bb P(\Sym^d\cal V)\bigr)\supseteq\Eff^1\bigl(\bb P(\cal V)\bigr)$. 
By the duality of \cite{bdpp13}, we also deduce
$\nu_{d*}\Mov_1\bigl(\bb P(\cal V)\bigr)\supseteq\Mov_1\bigl(\bb P(\Sym^d\cal V)\bigr).$
 
To see these, fix a very ample divisor $H$ on $X$ such that $\xi+\rho^*H$ is very ample on $\bb P(\cal V)$.
If $\delta$ is a $\bb Q$-Cartier $\bb Q$-divisor on $X$ such that 
$\xi+\rho^*\delta$ is pseudo-effective on $\bb P(\cal V)$, then 
for all $n\geq 0$ the class $n(\xi+\rho^*\delta)+(\xi+\rho^*H)$ is big.
For all $n\geq 0$ there then exist sufficiently divisible integers 
$a_n>0$ such that $\bigl|a_n\bigl((n+1)\xi+\rho^*(H+n\delta)\bigr)\bigr|$ is nonempty. Equivalently, 
$H^0\bigl(X,\ \Sym^{(n+1)a_n}\cal V\otimes\cal O_X(a_n(H+n\delta))\bigr)\neq 0$. 
The $d$-th power of a nonzero section is a nonzero section of
$\Sym^{d(n+1)a_n}\cal V\otimes\cal O_X(da_nH+dna_n\delta)$. 
In characteristic zero, this sheaf is a direct summand of 
$\Sym^{(n+1)a_n}\Sym^d\cal V\otimes\cal O_X(da_nH+dna_n\delta)$.
We deduce that $(n+1)a_n\xi_d+\rho_d^*(da_nH+dna_n\delta)$ is effective.
The ray that this spans in $N^1(\bb P(\Sym^d\cal V))$ approaches 
the span of $\xi_d+d\rho_d^*\delta$ as $n$ grows; therefore, $\xi_d+d\rho_d^*\delta$ is pseudo-effective.
Note that $\nu_d^*(\xi_d+d\rho_d^*\delta)=d(\xi+\rho^*\delta)$. 
\end{rmk}

\begin{lem}[Determinants]\label{lem:comparetodet}If $\cal V$ is locally free of
  rank $r$, then for all $x \in X$, we have
  \[\sh(\cal V;x)\leq\frac 1{r}\sh(\det\cal V;x).\]
\end{lem}

\begin{proof}Immediate from Corollary \ref{cor:shviacurves} and from $\mu_{\min}(\nu^*\cal V)\leq\mu(\nu^*\cal V)=\frac{\deg_{C'}(\nu^*\cal V)}{r}=\frac{\det\cal V\cdot C}{r}.$
\end{proof}

\begin{lem}[Tensor products]\label{lem:tensorproducts}Let $\cal V$ and $\cal V'$ be (twisted) coherent sheaves on $X$. Then 
\[\sh(\cal V\otimes\cal V';x)\geq \sh(\cal V;x)+\sh(\cal V';x)\]
for all $x\in X$. If $X$ is a curve, or if $\cal V'=\cal V$, then equality holds.
\end{lem}

\begin{proof}Corollary \ref{cor:shviacurves} allows to reduce to the case of possibly singular curves. By normalizing we can assume that $X$ is a smooth projective curve.
Pulling back by a sufficiently large iteration of the Frobenius, we may assume that $\overline{\mu}_{\rm min}=\mu_{\rm min}$ for all the (finitely many) sheaves involed.
Then in fact we claim
\begin{equation*}\mu_{\rm min}(\cal V\otimes\cal V')=\mu_{\rm min}(\cal V)+\mu_{\rm min}(\cal V').
\end{equation*}
Up to twisting, we may assume $\mu_{\rm min}(\cal V)=\mu_{\rm min}(\cal V')=0$, so $\cal V$ and $\cal V'$ are nef. 
Then $\cal V\otimes\cal V'$ is also nef (cf.\ \cite[Theorem 6.2.12]{laz04}), hence $\mu_{\min}(\cal V\otimes\cal V')\geq 0$.
If $\cal Q$ and $\cal Q'$ are quotients of slope 0 of $\cal V$ and $\cal V'$ respectively,
then $\cal Q\otimes\cal Q'$ is a quotient of slope 0 of $\cal V\otimes\cal V'$, giving the remaining inequality 
$\mu_{\rm min}(\cal V\otimes\cal V')\leq 0$.
\end{proof}

Note that equality on curves does not lead to equality in arbitrary dimension in general, since the Seshadri constants $\sh(\cal V;x)$ and $\sh(\cal V';x)$ could be 
approximated on different curves through $x$. We observe this below already for line bundles.

\begin{ex}On $X\coloneqq  \bb P^1\times\bb P^1$, we have $\sh(\cal O(1,0);x)=\sh(\cal O(0,1);x)=0$ for all $x\in X$, since the line bundles in question are nef 
and have trivial restrictions on the fibers of the respective natural projection. On the other hand, as in \cite[Example 5.1.7]{laz04}, we find $\sh(\cal O(1,1);x)=1$ for all $x\in X$. 
\qed
\end{ex}

\begin{cor}\label{cor:ampletensor}Let $X$ be a projective scheme over an algebraically closed field, and let $\cal V$ and $\cal V'$ be (twisted) 
sheaves on $X$. Assume that $\cal V$ is ample (resp.\ nef), and that $\cal V'$ is nef.
Then $\cal V\otimes\cal V'$ is ample (resp.\ nef). Furthermore, all Schur functors
$S_{\lambda}\cal V$ are ample (resp. nef), 
where $\lambda$ is a partition of some positive integer.\footnote{Here, $S_{\lambda}\cal V$, where
$\lambda=(\lambda_1\geq\ldots\geq\lambda_r\geq 0)\vdash n$, is understood as a
quotient of $\Sym^{\lambda_1}\cal V\otimes\ldots\otimes\Sym^{\lambda_r}\cal V$
as in \cite[Chapter 8.3, Example 10]{FultonYoung}.}
\end{cor}
\noindent Compare with \cite[Corollary 6.1.6]{laz042} and \cite{Barton71}.
\begin{proof}Immediate from Lemma \ref{lem:tensorproducts} and Theorem \ref{seshadriample} (resp. Remark \ref{rmk:easyseshadri}.(a)).
For the last part, use the construction of $S_{\lambda}\cal V$ as quotient of $\Sym^{\lambda_1}\cal V\otimes\ldots\otimes\Sym^{\lambda_r}\cal V$.
This reduces the problem to showing that $\Sym^n\cal V$ is ample (resp. nef)
if $\cal V$ is ample (resp. nef). 
This follows from Lemma \ref{lem:homogeneous} and from Theorem \ref{seshadriample} (resp. Remark \ref{rmk:easyseshadri}.(a)).
\end{proof}

\begin{lem}\label{lem:seshadrisequences}Let $\cal K\to\cal V\to\cal Q\to 0$ be an exact sequence of (twisted) coherent sheaves on $X$.
Then, we have 
\[\sh(\cal V;x)\geq\min\bigl\{\sh(\cal K;x),\sh(\cal Q;x)\bigr\}\] 
for all $x\in X$.
In particular, if $\sh(\cal K;x)\geq\sh(\cal Q;x)$, then $\sh(\cal V;x)=\sh(\cal Q;x)$.
Furthermore, if $\cal V=\cal K\oplus\cal Q$, then $\sh(\cal V;x)=\min\{\sh(\cal K;x),\sh(\cal Q;x)\}$.
\end{lem}
\begin{proof}By Corollary \ref{cor:shviacurves}, as above, we can assume that $X$ is a smooth curve, and that 
$\overline{\mu}_{\rm min}=\mu_{\rm min}$ for all the sheaves involved. 
Let $\cal V\twoheadrightarrow A$ be the quotient of minimal slope in the Harder--Narasimhan filtration of $\cal V$.
In particular, $A$ is semistable. 
If the induced map $\cal K\to A$ is nonzero, then its image has slope at most $\mu(A)=\mu_{\rm min}(\cal V)$, and $\mu_{\rm min}(\cal K)\leq\mu_{\rm min}(\cal V)$.
If $\cal K\to A$ is zero, then we obtain an induced nonzero map $\cal Q\to A$ and argue as before. 
The last part follows from Lemma \ref{lem:shquot}.
\end{proof}

\subsection{Pseudo-effectivity}
Using results from \cite{bdpp13} (which hold in arbitrary characteristic by
\cite[Section 2.2]{fl13z}), we show that
Seshadri constants for non-pseudo-effective divisors are negative.
\begin{lem}\label{lem:nonpsefnegative} Let $X$ be a projective variety of dimension $n$ over an
algebraically closed field, and let $L\in N^1(X)$ be an $\bb R$-Cartier $\bb R$-divisor class
outside the pseudo-effective cone $\Eff^1(X)$. Then, $\sh(L;x)=-\infty$ for general $x\in X$.
Furthermore, $\sh(L;x)<0$ for all $x$.
\end{lem}

\begin{proof}By \cite[Theorem 2.2]{bdpp13}, there exists a birational model $f\colon X'\to X$ and 
ample divisor classes $H_1,\ldots,H_{n-1}$ on $X'$ such that 
$L\cdot f_*(H_1\cdot\ldots\cdot H_{n-1})<0$.
Since there exist complete intersection curves through every point of $X'$, 
their images pass through every point of $X$. 
This implies $\sh(L;x)<0$ for all $x$.

Let $x \in X$ be a point where $f$ is an isomorphism, and denote the
inverse image of $x$ in $X'$ also by $x$.
By Bertini's theorem for Hilbert--Samuel multiplicity \cite[Proposition
4.5]{dFEM03}, for $m \gg 0$ there exist complete intersection curves $C_m$ of
members of $\lvert mH_i \rvert$ passing through $x$ with multiplicity $\mult_x
X$.
Then, we have $\sh(L;x)\leq\frac{L\cdot C_m}{\mult_x C_m}=\frac{L\cdot
C_m}{\mult_xX}$, and the right-hand side tends to $-\infty$ as $m\to \infty$.
\end{proof}

\begin{cor}\label{cor:psef}
  Let $X$ be a projective variety of dimension $n$ over an
  algebraically closed field.
  If $\cal O_{\bb P(\cal V)}(1)$ is not pseudo-effective, then $\sh(\cal V;x)<0$ for all $x\in X$.
\end{cor}
\begin{proof}Immediate from Lemma \ref{lem:nonpsefnegative} and the negative case of Remark \ref{rmk:compareless1}.
\end{proof}
\subsection{Semicontinuity}\label{section:semi}
We end this section with two semicontinuity results.
The first concerns semicontinuity in the $\bb R$-twists $\lambda$ for a twisted
sheaf of the form $\cal V\langle \lambda \rangle$, which is a consequence of our
functoriality results.
\begin{cor}\label{cor:seshadriconvex}Let $X$ be a projective scheme over an algebraically closed field, and fix a closed point $x\in X$. Let $\cal V$ be a coherent sheaf, with $\sh(\cal V;x)>-\infty$. 
Let $h$ be an ample divisor class on $X$.
Consider the function $\epsilon(t)\coloneqq  \varepsilon(\cal V\langle th\rangle;x)$. 
Then, $\epsilon$ is nondecreasing, continuous at all $t>0$ and lower-semicontinuous at $t=0$.
\end{cor}
\begin{proof}Lemma \ref{lem:tensorproducts} and homogeneity for divisors imply 
  \[\epsilon(t)\geq\epsilon(t')+(t-t')\sh(h;x)>\epsilon(t')\] for all $t>t'\geq 0$. Furthermore
\[\epsilon(t)+\epsilon(t')\leq\sh\bigl(\cal V\otimes\cal V\bigl\langle
(t+t')h\bigr\rangle;x\bigr)=2\sh\biggl(\cal
V\biggl\langle\frac{t+t'}2h\biggr\rangle;x\biggr)=2\epsilon\biggl(\frac{t+t'}2\biggr).\]
Finite concave functions are continuous on open intervals. Lower-semicontinuity follows because $\epsilon$ is nondecreasing. 
\end{proof}
Next, we prove the following semicontinuity result for Seshadri constants in
smooth families.

\begin{prop}[Semicontinuity of Seshadri constants]\label{prop:semicont}
Let $T$ be a smooth connected variety over an uncountable algebraically closed field.
Let $p\colon \scr X\to T$ be a smooth projective family of varieties with
connected fibers and a section $T\to\scr X$ which maps $t\mapsto x_t\in X_t\coloneqq   p^{-1}\{t\}$. 
Let $\scr V$ be a locally free sheaf on $\scr X$, and denote $\cal V_t$ the corresponding restriction to $X_t$. 

Let $\epsilon\geq 0$, and let $t_0\in T$ such that $\cal V_{t_0}$ is nef and
$\sh(\cal V_{t_0};x_{t_0})\geq\epsilon$.
Then, $\sh(\cal V_{t};x_t)\geq\epsilon$ for very general $t\in T$.

In particular, under the positivity assumptions above, 
the Seshadri constants are constant outside an at most countable union 
of proper closed subsets, on which they may decrease.

The same results work in the more general setting of a smooth projective morphism $\rho:\scr Y\to\scr X$ of $T$-schemes with $\rho$-ample polarization $\xi$. 
\end{prop}

\begin{proof}Let $\rho:\bb P(\scr V)\to\scr X$ be the bundle map with fiberwise restrictions $\rho_t:\bb P(\cal V_t)\to X_t$. If the conclusion fails, then standard relative Hilbert scheme arguments produce
a scheme of finite type $H$ with a dominant morphism $f:H\to T$ (by restriction to a closed subset we may assume that $f$ is generically finite) and a relative flat curve $\scr C\subset H\times_T\bb P(\scr V)$ over $H$ such that the fibers
$C_h\subset \bb P(\cal V_{f(h)})$ are irreducible, and moreover in $\cal C_{\cal V_{f(h)},x_{f(h)}}$ for all $h\in H$. Furthermore $\frac{\xi_{f(h)}\cdot C_h}{\mult_{x_{f(h)}}\rho_{t*}C_h}<\epsilon$. 

Let $Y\subset\bb P(\scr V)$ be the closure of $(f\times_T\bb P(\scr V))(\scr C)=\cup_{h\in H}C_h\subset\bb P(\scr V)$. For any $t\in T$, denote by
$[Y_t]$ the Chow class of the restriction $Y|_{\bb P(\cal V_t)}$ in the sense of \cite[Chapter 8]{fulton84}.
This is an effective curve class (even if the scheme theoretic fiber $Y_t$ may have dimension greater than 1). See \cite[Lemma 4.10]{fl16sw} for details. 
For very general $t\in T$, the class $[Y_t]$ is represented by the fundamental cycle of the scheme theoretic $Y_t$ which is just the sum (with multiplicity) of the finitely many $C_h$ with $h\in f^{-1}t$. 
By abuse, we write $[Y_t]=Y_t=\sum_{h\in f^{-1}t}C_h$ in this case.

For $t\in T$ very general, let $Z_0+Z'_0$ be a flat degeneration over $t_0$ of the restriction of $Y$ over some irreducible curve $T'\subset T$ connecting $t$ and $t_0$. 
In fact, $Z_0+Z_0'$ is the fundamental cycle of the fiber over $t_0$ of the irreducible component of $Y_{T'}$ that dominates $T'$. 
Here $Z'_0$ is the part that does not come from $\cal C_{\cal V_{x_{t_0}},x_{t_0}}$. 
Since multiplicity is upper semicontinuous in families, and $\rho_{t_0*}Z_0'$
does not have $x_{t_0}$ in its support, we have
\[\mult_{x_{t_0}}\rho_{t_0*}Z_0\geq\mult_{x_t}\rho_{t*}[Y_t].\]
Since $\xi_{{t_0}}$ is nef, we have $\xi_{{t_0}}\cdot[Y_{t_0}]\geq\xi_{x_{t_0}}\cdot Z_0$.
 We reach the contradiction
\begin{align*}
\epsilon\leq\frac{\xi_{{t_0}}\cdot
Z_0}{\mult_{x_{t_0}}\rho_{t_0*}Z_0}\leq\frac{\xi_{{t_0}}\cdot[Y_{t_0}]}{\mult_{x_{t_0}}\rho_{t_0*}Z_0}&\leq
\frac{\xi_{{t}}\cdot[Y_{t}]}{\mult_{x_{t}}\rho_{t*}Y_t}=\frac{\xi_{t}\cdot
\sum_{h\in f^{-1}t}C_h}{\sum_{h\in f^{-1}t}\mult_{x_t}\rho_{t*}C_h}\\
&\leq\max_{h\in f^{-1}t}\frac{\xi_{t}\cdot
C_h}{\mult_{x_t}\rho_{t*}C_h}<\epsilon.\qedhere
\end{align*}
\end{proof}

\begin{rmk}The only step in the proof of Proposition \ref{prop:semicont} where the nefness of $\cal V_{t_0}$ is used is in the inequality $\xi_{t_0}\cdot Z_0\leq \xi_{t_0}\cdot[Y_{t_0}]$. What could go wrong is $Z_0'$ 
having components in $\bb P(\cal V_{t_0})$ that do not intersect the fiber over $x_{t_0}$. Recall that $\xi_{t_0}$ is positive on curves fully contained in the fiber $\rho_{t_0}^{-1}x_{t_0}$.
For the conclusion of Proposition \ref{prop:semicont}, it would be enough to ask that $Z_0'$ has no components contained in the non-nef locus of $\xi_{t_0}$.

In the absence of the nefness condition (and of the positivity of $\epsilon$), \cite[Example 3.15]{ful17} observes that this form of lower semicontinuity fails already for line bundles on toric surfaces.
\end{rmk}

\section{Products of curves}
We now come to our first application of our new formalism for Seshadri
constants.
We start by setting the following notation for the rest of this section.
\begin{notn}\label{notn:prodcurves}
  Let $C$ be a smooth projective curve of genus $g$ over $\bb C$.
  Let $X = C \times C$, and let $p$ and $q$ denote the projections onto each
  factor.
  Let $f_1$ denote the class of the fiber of $p$ and $f_2$ the class of a fiber
  of $q$.
  Denote by $\delta$ the class of the diagonal $\Delta$.
\end{notn}

It is a tantalizingly open problem to understand the nef cone of $C\times C$, 
even in the symmetric slice given by intersecting with the span of $f_1+f_2$ and $\delta$.
The classes $f_1$ and $f_2$ are clearly nef. If $a,b,c\geq 0$,
then $af_1+bf_2+c\delta$ is nef if and only if 
$(af_1+bf_2+c\delta)\cdot\delta=a+b-c(2g-2)\geq 0$.
For example, $(g-1)f_1+(g-1)f_2+\delta$ is the pullback 
of the theta polarization on the Jacobian of $C$ via the difference map
$C\times C\to{\rm Jac}(C): (x,y)\mapsto x-y$. 

The class $af_1-bf_2+\delta$ is never nef when $g\geq 2$. 
If it were, then by symmetry 
$bf_1-af_2+\delta$ would also be nef. Intersecting with 
$\delta$, we get $\pm(a-b)\geq 2g-2>0$, which is impossible. 
The class $af_1-bf_2-\delta$ is not nef either (or even pseudo-effective), because it has negative intersection with $f_1$. 
After removing $\delta$ from the negative part of the Zariski 
decomposition of $af_1-bf_2+\delta$, one shows that these
divisors are not pseudo-effective either when $b>0$. 

It remains unclear when the classes $af_1+bf_2-\delta$ are nef. 
By intersecting with $f_1$ and $f_2$, we get $a\geq 1$ and $b\geq 1$
as necessary conditions. By considering the self intersection,
we also have $a>1$ and $b\geq 1+\frac g{a-1}$.
Conjecture \ref{conj:prodcurvesintro} predicts that if $C$ is very general and
$g \gg 0$, then the divisor class
  \[
    (\sqrt{g}+1)(f_1+f_2) - \delta
  \]
  is nef.
Note that this divisor has self-intersection zero as in the famous Nagata conjecture.
In fact, \cite{cknagata,Ross} prove that the Nagata conjecture implies Conjecture \ref{conj:prodcurvesintro}.
One could extend the conjecture to the non-symmetric divisors with zero self-intersection
\begin{equation}\label{eq:conjcurves}
  af_1 + \biggl( 1 + \frac{g}{a-1} \biggr) f_2 - \delta
\end{equation}
for all $a > 1$.

\begin{rmk}[Vojta's divisors]
Inspired by \cite{vojta}, \cite[Proposition 3.2]{ashwath}
proves that if $r,s>0$, then 
$(\sqrt{(g+s)r^{-1}}+1)f_1+(\sqrt{(g+s)r}+1)f_2-\delta$
is nef if $r\geq \frac{(g+s)(g-1)}s$.\footnote{Note that \cite{ashwath} denotes our class $\delta-f_1-f_2$ by $\delta$.}

Setting $a=\sqrt{(g+s)r}+1$ and $r=\frac{(g+s)(g-1)}s$, we deduce the nefness of the divisor
\begin{equation}\label{eq:Vojta}
af_1+\left(1+\frac{2g}{a-1+\sqrt{(a-1)^2-4g(g-1)}}\right)f_2-\delta,
\end{equation}
for $a\geq 1+2\sqrt{g(g-1)}$, e.g., $a\geq 2g$.
These are close to the conjectural bound \eqref{eq:conjcurves} for large $a$, but never equal to it when $a>1$.
Setting $b=\sqrt{(g+s)r^{-1}}+1$ and $r=\frac{(g+s)(g-1)}s$, we deduce the nefness of
\begin{equation}\label{eq:Vojta2}
\biggl(\frac g{b-1}+(b-1)(g-1)+1\biggr)f_1+bf_2-\delta.
\end{equation}
\end{rmk}

To demonstrate a possible approach to this question via Seshadri constants, 
we start by showing that over curves, the relative
Seshadri constants are approximated by Seshadri constants of locally free sheaves.

\begin{prop}\label{prop:seshadriapprox}
Let $\rho\colon Y\to X$ be a surjective morphism of projective schemes over an algebraically closed field. 
Let $\cal L$ be a $\rho$-ample line bundle on $Y$.
Put $\cal F_n\coloneqq  \rho_*(\cal L^{\otimes n})$.
We then have
\[
  \sh(\cal L;x)\geq\limsup_{n\to\infty}\frac{\sh(\cal F_n;x)}n
\]
for all $x\in X$. When $X$ is an irreducible curve, then $\sh(\cal L;x)=\lim\limits_{n\to\infty}\frac{\sh(\cal F_n;x)}n$.
\end{prop}

\begin{proof}Since $\cal L$ is $\rho$-ample, 
for large $n$, we have surjections $\rho^*\cal F_n\twoheadrightarrow \cal L^{\otimes n}$  (cf. \cite[Theorem 1.7.6.(iii)]{laz04}) inducing closed immersions 
$Y\hookrightarrow\bb P(\cal F_n)$ such that $\cal O(1)|_Y=\cal L^{\otimes n}$. 
From Lemma \ref{lem:shquot}, it follows that $\sh(\cal L;x)\geq\frac{\sh(\cal F_n;x)}n$
and $\sh(\cal L;x)\geq\limsup_{n\to\infty}\frac{\sh(\cal F_n;x)}n$.

If $X$ is an irreducible curve, let $f$ be the class of a general fiber of $\rho$. 
From Remark \ref{rmk:curves}, 
\begin{equation}\label{fuj1}\sh(\cal L;x)=\frac 1{\mult_xX}\cdot\sup\bigl\{t\st
c_1(\cal L)-tf\in\Nef^1(Y)\bigr\}.\end{equation}
In particular, $\mult_xX\cdot\sh(\cal L;x)$ is independent of $x$. 
From the projection formula we have $\frac{f\cdot C}{\mult_x\rho_*C}=\frac 1{\mult_xX}$ for all $C\in\cal C_{\rho,x}$, hence
$\sh(\cal L\langle -tf\rangle;x)=\sh(\cal L;x)-\frac t{\mult_xX}$ for all $t\in\bb R$.
Now consider an arbitrary rational number $\frac ab<\sh(\cal L;x)$ with $a,b$ integers and $b>0$. 
From homogeneity (Lemma \ref{lem:homogeneous}), from the Seshadri ampleness criterion (Theorem \ref{seshadriample}), and from \eqref{fuj1},
we obtain that $\cal L^{\otimes b}(-af)$ is ample and $\frac{\sh(\cal L^{\otimes b}(-af);x)}b=\sh(\cal L;x)-\frac a{b\cdot\mult_xX}$.
Furthermore, from the projection formula, $\rho_*\bigl((\cal L^{\otimes
b}(-af))^{\otimes m}\bigr)=\cal F_{bm}(-am)$, where $\cal O(1)\coloneqq  \cal O_X(x)$.
With similar arguments, $\sh(\cal F_{bm}(-am);x)=\sh(\cal F_{bm})-\frac{am}{\mult_xX}$.

Up to twisting, we may assume that $\cal L$ is ample. 
By Lemma \ref{lem:CM} below, $\sh(\cal F_n;x)\geq 0$ for sufficiently large $n$.
Since $\cal L$ is relatively ample, there exists $n_0\geq 0$ such that the natural map 
$\cal F_n\otimes\cal F_m\to\cal F_{n+m}$ is surjective for all $n\geq n_0$ and $m\geq  1$ \cite[Example 1.8.24.(ii)]{laz04}.
From Lemma \ref{lem:tensorproducts} we deduce that $\sh(\cal F_n;x)$ is an eventually superadditive 
sequence of nonnegative real numbers.
Fekete's Lemma proves that $\lim_{n\to\infty}\frac{\sh(\cal F;x)}n$ exists.

For the inequality $\sh(\cal L;x)\leq\lim_{n\to\infty}\frac{\sh(\cal F_n;x)}n$, 
since we may replace $\cal L$ by $\cal L^{\otimes b}(-af)$ for any integers $a,b$ with $b\geq 1$,
it is enough to prove that if $\cal L$ is ample, then $\lim_{n\to\infty}\frac{\sh(\cal F_n;x)}n\geq 0$. This is clear by Lemma \ref{lem:CM} below, since $\cal F_n$ is nef for $n$ sufficiently large. 
\end{proof}

\begin{lem}\label{lem:CM}Let $\rho:Y\to X$ be a morphism of projective schemes, and let $\cal L$ be an ample invertible sheaf on $Y$.
Let $\cal F$ be a coherent sheaf on $X$.
Then $\cal F\otimes\rho_*\cal L^{\otimes n}$ is ample and globally generated for all $n$ sufficiently large.
\end{lem}
\begin{proof}
Let $A$ be a very ample divisor on $X$ such that there exists a surjection
$\bigoplus\cal O_X(-A)\twoheadrightarrow\cal F$.
Since ampleness and global generation descend to quotients, it is enough
to prove the lemma for $\cal F=\cal O_X(-A)$.
With the usual arguments of Castelnuovo--Mumford regularity \cite[Theorem 1.8.5]{laz04}, it is enough to prove
that if $A$ is a very ample divisor on $X$, then 
$\rho_*\cal L^{\otimes n}$ is $-2$-regular with respect to $A$,
i.e., $H^i\bigl(X;\rho_*\cal L^{\otimes n}(-(2+i)A)\bigr)=0$ for all $i>0$ for all $n$ sufficiently large. 
This is because in this case $\rho_*\cal L^{\otimes n}(-2A)$ is globally generated, hence $\rho_*\cal L^{\otimes n}(-A)$ is ample and globally generated.

Since $\cal L$ is ample, it is in particular also $\rho$-ample. 
Hence for $n$ large, we have $R^i\rho_*\cal L^{\otimes n}=0$ for all $i>0$. 
The Leray spectral sequence and the projection formula show that 
$H^i\bigl(X;\rho_*\cal L^{\otimes n}(-(2+i)A)\bigr)=H^i\bigl(Y;\cal L^{\otimes n}\otimes\rho^*(-(2+i)A)\bigr)$.
The ampleness of $\cal L$ and Serre vanishing show that these cohomology groups are $0$.
\end{proof}

\begin{cor}\label{cor:curveapprox}Let $\rho\colon X\to C$ be a flat morphism between projective varieties with $C$ a nonsingular curve.
Let $\cal L$ be a $\rho$-ample line bundle, and let $f$ be the class of a fiber of $\rho$.
Then 
\[\sup\bigl\{t\st c_1(\cal L)-tf\ \text{is nef}\bigr\}=\lim_{n\to\infty}\frac{\overline\mu_{\min}(\rho_*\cal L^{\otimes
n})}n.\]
\end{cor}

\begin{defn}If $L$ is a Cartier divisor on $C$ and $i\geq 0$, denote 
\[R^{i-1}(L)\coloneqq q_*\bigl(p^*\cal O_C(L)\otimes\cal O_X(-i\Delta)\bigr).\]
\end{defn}

We now prove the following result in the spirit of Conjecture
\ref{conj:prodcurvesintro}.
\begin{thrm}\label{thrm:prodcurves}
  Use notation as in Notation \ref{notn:prodcurves}. 
\begin{enumerate}
\item[\emph{(i)}] If $a>1$ is a rational number, then the class \eqref{eq:conjcurves}
$af_1+\bigl(1+\frac g{a-1}\bigr)f_2-\delta$ is nef if and only if 
the sheaves $R^{n-1}(nL)$ are \emph{asymptotically semi-stable}
\footnote{It makes sense to ask if $R^{n-1}(nL)$ is (semi)stable for large divisible $n$. See also \cite[Conjecture 4.2]{elstable}.}
, i.e.,
\[
\lim_{n\to\infty}\frac 1n\mu_{\rm min}\bigl(R^{n-1}(nL)\bigr)=\lim_{n\to\infty}\frac 1n\mu\bigl(R^{n-1}(nL)\bigr),
\]
where $L$ is $\bb Q$-divisor on $C$ with $\deg L=a$, and $n$ is sufficiently divisible. 
\item[\emph{(ii)}] If $g\geq 3$ and if $C$ is general, then the divisor class
  \[
    df_1+\left(1+\frac g{d-g}\right)f_2-\delta
  \]
  is nef for all integers $d\geq \lfloor 3g/2\rfloor +3$.
\item[\emph{(iii)}] If $g\geq 3$ and $C$ is very general, then the divisors $af_1+bf_2-\delta$
are nef for all $(a,b)$ in the convex hull of
\begin{align*}
  \MoveEqLeft[5]\biggl\{
    \biggl(\frac{g}{b-1}+(b-1)(g-1)+1,b\biggr)
    \biggm\vert b \in (1,2]
  \biggr\}
  \cup \bigl\{(a,b)\bigm\vert a+b=2g+2,\ a,b\geq 2\bigr\}\\
  &\cup \biggl\{
    \biggl(a,\frac{g}{a-1}+(a-1)(g-1)+1\biggr)\biggm\vert a\in(1,2]
  \biggr\}\\
  &\cup \biggl\{
    \biggl(2g-k,1+\frac{g}{g-k}\biggr)\ ,\ \biggl(1+\frac{g}{g-k},2g-k\biggr)\biggm\vert
    k\in\biggl\{1,2,\ldots,\biggl\lfloor\frac{g-5}2\biggr\rfloor\biggr\}
  \biggr\}\\
  &\cup \biggl\{\biggl(\frac{g}{\lfloor \sqrt g\rfloor}+1,\frac{g}{\lfloor
  \sqrt g\rfloor}+1\biggr)\biggr\}.
\end{align*}
\end{enumerate}
\end{thrm}

The divisors in (ii) improve Vojta's examples \eqref{eq:Vojta} in the range $\lfloor 3g/2\rfloor+3\leq d<2g$, which is nonempty when $g\geq 7$.
This range is responsible for the fourth set in the union in (iii).

\begin{proof}(i)
By considering the $q$-ample class $af_1-\delta$, Corollary
\ref{cor:curveapprox} reduces the nefness 
of $af_1+\bigl(1+\frac g{a-1}\bigr)f_2-\delta$ to proving that for very divisible $n$, the sequence of normalized slopes
$\frac{\mu_{\min}(R^{n-1}(nL))}n$ limits to $-1-\frac g{a-1}$.
Since $a>1$, for large divisible $n$ we have exact sequences 
\[0\to R^{n-1}(nL)\to H^0(C,\cal O(nL))\otimes\cal O_C\to  P^{n-1}\cal O(nL)\to
0.\]
Recall that if $\cal L$ is a line bundle, then $P^{n-1}\cal L$ denotes the bundle of principal parts $q_*(p^*\cal L\otimes\cal O_{n\Delta})$. 
It is a rank $n$ vector bundle with a natural filtration with quotients $\cal L,\ \cal L\otimes \omega_C,\ \ldots\ ,\ \cal L\otimes \omega^{\otimes(n-1)}_C$. 
From this, one computes
$\mu(R^{n-1}(nL))=-n(1+\frac{ng}{na+1-g-n})$. As $n$ grows, $\frac 1n\mu(R^{n-1}(nL))$ approaches $-\bigl(1+\frac{g}{a-1}\bigr)$.
In particular, the nefness of $af_1+(1+\frac g{a-1})f_2-\delta$ is equivalent to the asymptotic semistability of $R^{n-1}(nL)$.

(ii) Assume first $d\geq 2g+2$.
  Let $\cal L$ be a line bundle of degree $d$ on $C$. 
By $M_{\cal L}=R^0(\cal L)=q_*(p^*\cal L\otimes\cal O_X(-\Delta))$ denote the kernel of the evaluation $H^0(C,\cal L)\otimes\cal O_C\to \cal L$.
As in \cite[\S 4]{elstable}, one finds a surjection $M_{\cal
L}^{\otimes n}\twoheadrightarrow R^{n-1}(n\cal L)$ coming from the surjections
$(H^0(C;\cal L(-x)))^{\otimes n}\twoheadrightarrow H^0(C;\cal L^{\otimes n}(-nx))$
for all $x\in X$. Via the semistability of $M_{\cal L}$ (cf.\ \cite[Proposition 3.2]{elstable}) for $d\geq 2g$, this shows that 
\[
  \overline\mu_{\min}\bigl(p_*((q^*\cal L\otimes\cal O_X(-\Delta))^{\otimes
  n})\bigr)\geq\mu\bigl(M_{\cal L}^{\otimes n}\bigr)=n\cdot\mu\bigl(M_{\cal
  L}\bigr)=-n\biggl(1+\frac g{d-g}\biggr),
\]
and the result then follows by Corollary \ref{cor:curveapprox}.

When $d=2g+1$, then $M_{\cal L}$ is still in fact stable. 
When $C$ is not hyperelliptic, then 
a general line bundle of degree $2g$ on $C$ 
is normally generated. 
There therefore exists a line bundle $\cal L$ of degree $2g+1$ such that $\cal L(-x)$ 
is normally generated for
general $x\in C$. The induced map $M_{\cal L}^{\otimes n}\to R^{n-1}(n\cal L)$ is generically surjective.
If $R^{n-1}(n\cal L)\twoheadrightarrow Q$ is a 
semistable quotient of positive rank,
this induces $M_{\cal L}^{\otimes N}\twoheadrightarrow Q'\subseteq Q$
with $Q'$ nonzero. By the semistability of $M_{\cal L}$ and of $Q$,
we obtain $\mu(Q)\geq\mu(Q')\geq n\cdot\mu(M_{\cal L})$.
Conclude as in the case $d\geq 2g+2$.

Assume now $d\geq\lfloor 3g/2\rfloor +3$. 
By \cite[p.~222, Theorem]{acgh}, a general divisor
of degree $\lfloor3g/2\rfloor+2$ is normally generated when $g\geq 3$.
For a general choice of a bundle $\cal L$ of degree $d$,
the divisors $\cal L(-x)$ are then normally generated for general $x\in C$.

\cite[Proposition 3.1]{EusenSchreyer} (and also D. Butler in unpublished work)
show that $N_{\cal L}:=M_{\cal L}^{\vee}$ (hence also $M_{\cal L}$) is stable
if $C$ is non-hyperelliptic, and if $\cal L$ is globally generated 
with ${\rm Cliff}(\cal L)\leq{\rm Cliff}(C)$ and $d\neq 2g$.
In our situation, when $\cal L$ is general
of degree $d$, then ${\rm Cliff}(\cal L)=d-2(h^0(\cal L)-1)=2g-d$.
Since $C$ is general, it is not hyperelliptic, 
and ${\rm Cliff}(C)=\lfloor (g-1)/2\rfloor$ 
by \cite{acgh}. Note that $\lfloor (g-1)/2\rfloor>\lfloor (g-3)/2\rfloor
=2g-(\lfloor 3g/2\rfloor +3)\geq 2g-d$, therefore 
${\rm Cliff}(\cal L)<{\rm Cliff}(C)$. When $d\neq 2g$, we deduce from \cite[Proposition 3.1]{EusenSchreyer} 
that $M_{\cal L}$ is stable, and the argument concludes as above.
When $d=2g$, then $2gf_1+2f_2-\delta$ is the Vojta divisor \eqref{eq:Vojta} for $a=2g$.

(iii) The ``continuous'' part comes from Vojta's examples \eqref{eq:Vojta2} for $b\in(1,2]$, using the $(a,b)\leftrightarrow(b,a)$ symmetry, and the convexity of $\Nef(C\times C)$.
For example the line $a+b=2g+2$ is tangent to $a=\frac g{b-1}+(b-1)(g-1)+1$ at $(2g,2)$ and to $b=\frac g{a-1}+(a-1)(g-1)+1$ at $(2,2g)$.
The first finite set corresponds to the range $\lfloor 3g/2\rfloor+3\leq d<2g$, when (ii) is better than Vojta's examples \eqref{eq:Vojta}. 
The nefness of $\bigl(\frac{g}{\lfloor \sqrt{g}\rfloor}+1\bigr)(f_1+f_2)-\delta$ is the best known bound for Conjecture \ref{conj:prodcurvesintro}.
See \cite[Theorem 2]{kouvidakis}, \cite[Corollary 1.5.9]{laz04}.
Then, (iii) is a consequence of the convexity of $\Nef(X)$.
\end{proof}

\begin{ex}If $C$ is a very general curve of genus $g=7$, then (iii) already notices an improvement of (ii) for $d=2g$.
The tangent from $\bigl(2g-1,1+\frac g{g-1}\bigr)=\bigl(13,\frac{13}6\bigr)$ to the curve 
$\bigl(\frac 7{b-1}+6(b-1)+1,b\bigr)$ for $b\in(1,2]$
cuts the line $b=2$ at $a=13+2\frac{\sqrt 6}7\approx 13.699$, showing that $13.7f_1+2f_2-\delta$ is nef.
The old bound was $14f_1+2f_2-\delta$. The conjectural bound is $8f_1+2f_2-\delta$. 

The class $13f_1+\frac{13}6f_2-\delta$ is outside the convex span of the Vojta divisors and the Kouvidakis nef class $4.5(f_1+f_2)-\delta$.
Indeed the tangent from $(4.5,4.5)$ to the curve $\bigl(\frac 7{b-1}+6(b-1)+1,b\bigr)$ for $b\in(1,2]$ has slope approximately $-1/3.71$,
whereas the segment joining $(4.5,4.5)$ to $\bigl(13,\frac{13}6\bigr)$ has slope approximately $-1/3.64$, which is smaller.
\end{ex}

\begin{rmk}\cite{EN18} prove that $R^k(L)$ is semi-stable if 
\[\deg L =(k^2+2k+2)g+k.\]
With a strategy similar to Theorem \ref{thrm:prodcurves}.(ii), this gives a new proof that Vojta's divisors \eqref{eq:Vojta2}
are nef when $b=1+\frac 1{k+1}$ with $k\geq 0$ an integer. 
\end{rmk}

Corollary \ref{cor:curveapprox} extends to a more general setting:

\begin{prop}\label{prop:positivitybypushforward}Let $\rho:Y\to X$ be a morphism of projective schemes over an algebraically closed field. 
Let $\cal L$ be a $\rho$-ample line bundle on $Y$.
For $\cal F$ a coherent sheaf on $X$, and $H$ an ample line bundle on $X$, denote 
$\nu_H(\cal F)\coloneqq  \sup\{t \st \cal F\langle -tH\rangle\mbox{ is nef}\}.$
Then
\[\sup\bigl\{t\st c_1(\cal L)-t\rho^*H\mbox{ is nef}\bigr\}=\lim_{n\to\infty}\frac{\nu_H(\rho_*\cal L^{\otimes n})}n.\]
\end{prop}
\begin{proof}
The sequence $\nu_H(\rho_*\cal L^{\otimes n})>-\infty$ is superadditive by $\rho$-ampleness, hence the limit exists by Fekete's Lemma.
Since $\cal L$ is $\rho$-ample, for sufficiently large $n$, we have inclusions $Y\hookrightarrow\bb P(\rho_*\cal L^{\otimes n})$
such that $\cal O_{\bb P_X(\rho_*\cal L^{\otimes n})}(1)|_Y=\cal L^{\otimes n}$.
It follows that the inequality ``$\geq$'' holds.

For the reverse inequality, note as in Proposition \ref{prop:seshadriapprox} that both sides translate by $t_0$ when replacing 
$\cal L$ by $\cal L\langle t_0\rho^*H\rangle$ for $t_0\in\bb Q$ (with the
understanding that we only consider sufficiently divisible $n$ in the right-hand
side).
Without loss of generality, we may assume that $\cal L$ is ample on $Y$. 
As in Proposition \ref{prop:seshadriapprox}, we reduce to proving that $\rho_*\cal L^{\otimes n}$ is globally generated for large $n$,
which follows from Lemma \ref{lem:CM}.
\end{proof}

\begin{rmk}One could also try to approach Conjecture \ref{conj:prodcurvesintro}
by considering the difference map
$M\colon C\times C\to \operatorname{Pic}^0(C)$, which maps $(a,b)\mapsto a-b$, and the $M$-ample class $-\delta$. 
Let $\theta$ be a principal polarization on $\operatorname{Pic}^0C$.
The conjecture is equivalent to 
$$\sup\bigl\{t\st -\delta-tM^*\theta\mbox{ is nef}\bigr\}=-\frac 1{\sqrt g}.$$
(From \cite[Example 1.5.14]{laz04},  we have $M^*\theta=(g-1)(f_1+f_2)+\delta$. Asking that $-\delta-tM^*\theta$ be a scalar multiple of $(\sqrt g+1)(f_1+f_2)-\delta$ is equivalent to $\det\left|\begin{matrix} -t(g-1)& -1-t\\ \sqrt g+1& -1\end{matrix}\right|=0$.)
\end{rmk}

\begin{rmk}For a Seshadri constant approach via the difference map $M\colon C\times C\to\operatorname{Pic}^0C$, one would have to prove that
$$\sh(f_1+f_2;o)=\frac 1{\sqrt g+1},$$ where $o\in\operatorname{Pic}^0C$ is the origin.

To see the equivalence of this with Conjecture \ref{conj:prodcurvesintro}, note that $M$ factors through the blow-up of $o$, and $\delta$ is the pullback of the exceptional divisor.
If $T\subset C\times C$ is a curve, then $\frac{(f_1+f_2)\cdot T}{\mult_oM_*T}=\frac{(f_1+f_2)\cdot T}{\delta\cdot T}$.

This argument also shows that $\sh(-\delta;o)=-1$.
\end{rmk}

\section{Tangent bundles}
Let $X$ be a smooth projective variety, and let $TX$ be the tangent sheaf. 
We are interested in the Seshadri constants of this bundle and in how they
recover some of the birational geometry of $X$. The motivation is given by the
following easy consequence of the Seshadri ampleness criterion (Theorem
\ref{seshadriample}) and Mori's
characterization of projective space \cite[V.3.3 Corollary]{kollarrational}.

\begin{cor}Let $X$ be a smooth projective variety.
If $\inf_{x\in X}\sh(TX;x)>0$, then $X\simeq\bb P^n$.
\end{cor}

\subsection{Examples}
We start by computing some examples.
\begin{ex}[Seshadri constants for $T\bb P^n$]\label{ex:tangentpn}
We have
\[
  \sh(T\bb P^n;x)= \begin{cases}
    2 & \text{if}\ n = 1;\\
    1 & \text{if}\ n \ge 2.
  \end{cases}
\]
(For $n\geq 2$, we have that $\bb P(T\bb P^n)$ sits naturally in $\bb P^n\times(\bb P^n)^{\vee}$ as the universal hyperplane $\sum_{i=0}^nx_iy_i=0$.
The class of the restriction $\cal O(1,1)|_{\bb P(T\bb P^n)}$ is $\xi$. Since $\cal O(0,1)$ is nef,  
$\frac{\xi\cdot C}{\mult_x\rho_*C}\geq\frac{\deg \rho_*C}{\mult_x\rho_*C}\geq 1$ for all $C\in\cal C_{T\bb P^n;x}$. 
The lower bound $1$ is achieved.
To see this, let $x\in\ell\subseteq H\subset\bb P^n$ be a line contained in a linear hyperplane $H$. 
Let $C\coloneqq  \ell\times\{[H]\}$ be the corresponding line in $\bb P^n\times(\bb P^n)^{\vee}$. 
It is contained in $\bb P(T\bb P^n)$ and $\frac{\xi\cdot C}{\mult_x\rho_*C}=\frac{\deg\ell}{\mult_x\ell}=1$.

For $\bb P^1$, we have $T\bb P^1=\cal O_{\bb P^1}(2)$ and the conclusion follows.
Note that in this case there are no horizontal lines $C$ as in the previous argument.
The restriction of the second projection $\bb P(T\bb P^1)\to\bb (\bb P^1)^{\vee}$ is an isomorphism.
)\qed
\end{ex}

\begin{ex}[Homogeneous varieties]\label{ex:tangenthomog}
If $X$ is a homogeneous variety (e.g., abelian or rational homogeneous space like a Grassmann variety or smooth quadric), not isomorphic to a projective space, 
then $TX$ is globally generated but not ample.
Since $X$ has a transitive algebraic group action, $\sh(TX;x)$ is independent of $x\in X$.
Then $\sh(TX;x)=0$ for all $x\in X$ by the Seshadri ampleness criterion (Theorem \ref{seshadriample}).
\end{ex}

\begin{ex}[Varieties of general type] Assume that $X$ is smooth projective variety over an algebraically closed field, with $K_X$ big (or even pseudo-effective, but not numerically trivial).
Then \[\sh(TX;x)=-\infty\quad \forall\ x\in X.\]
(Let $C_d$ be a smooth curve through $x$ with $\lim_{d\to\infty}K_X\cdot C_d=\infty$. 
General complete intersections through $x$ of large degree will do.
Then $\sh(-K_X;x)=-\infty$. Conclude by Lemma \ref{lem:comparetodet}.)\qed
\end{ex}

\begin{ex}[Calabi--Yau type manifolds]Assume that $X$ is a smooth projective variety
  of dimension $n$ over an algebraically closed field, with $K_X$ numerically
  trivial. Then \[\sh(TX;x)\leq 0\quad\forall\ x\in X.\]
(Indeed $\sh(TX;x)\leq\frac 1n\sh(\det TX;x)=0$.)\qed
\end{ex}

\begin{cor}[Uniruledness and Separably rationally connectedness (SRC) criterion]\label{cor:uniruled} Let $X$ be a smooth projective variety over an algebraically closed field. 
Assume there exists $x_0\in X$ such that $\sh(TX;x_0)>0$.
Then $X$ is uniruled, even SRC, and $\cal O_{\bb P(TX)}(1)$ is pseudo-effective.
\end{cor}
\begin{proof}The previous two examples show that $K_X$ is not pseudo-effective.
Then $X$ is uniruled by \cite{bdpp13} (whose results hold in arbitrary
characteristic by \cite[Section 2.2]{fl13z}). 
\par We now show that $X$ is SRC.
Since $\frac{1}{n\dim X} \sh(-K_X;x_0) \ge \sh(TX;x_0) > 0$ by Lemma
\ref{lem:comparetodet}, we see that $- K_X \cdot C > 0$ for every curve $C$
through $x_0$.
By bend and break \cite[II.5.14 Theorem]{kollarrational}, there therefore exists
a rational curve $D$ through $x_0$.
Since $\sh(TX;x_0)>0$, we see that $TX|_D$ is very free by Example
\ref{ex:curves}, and it follows that $X$ is SRC by \cite[IV.3.7 Theorem]{kollarrational}.

For the pseudo-effectivity statement, see Corollary \ref{cor:psef}.
\end{proof}

\begin{rmk}\label{rmk:criterionnonpositiveseshadri}The previous criterion is not a characterization of uniruled or SRC varieties. 
If $f\colon X\to Y$ is a smooth morphism of smooth projective varieties with positive dimensional fibers and $\dim Y>0$, we claim that $\sh(TX;x)\leq 0$ for all $x\in X$.
This applies in particular to Hirzebruch surfaces.
(Let $y\coloneqq  \pi(x)\in Y$. From the surjections 
\[TX\twoheadrightarrow f^*TY\twoheadrightarrow f^*TY|_{X_y}=\cal
O_{X_y}^{\oplus\dim Y},\]
by Lemma \ref{lem:shquot} we deduce 
$\sh(TX;x)\leq\sh(\cal O_{X_y}^{\oplus\dim Y};x)=0$.)\qed  
\end{rmk}

\subsection{Characterizations of projective space}
In particular cases, we can say something stronger than Corollary \ref{cor:uniruled}
when $\sh(TX;x_0)>0$ for a point $x_0 \in X$.
\begin{prop}[Fano manifolds]\label{prop:Fanos}
  Let $X$ be a smooth projective variety over an algebraically closed field $k$.
  Suppose that one of the following conditions holds:
  \begin{enumerate}
    \item[\textup{(1)}] $X$ is Fano and some $x_0\in X$ verifies $\sh(TX;x_0)>0$;
    \item[\textup{(2)}] $\operatorname{char} k = 0$ and a \emph{general} point $x_0\in X$
      verifies $\sh(TX;x_0)>0$.
  \end{enumerate}
  Then, $X\simeq\bb P^n$.
\end{prop}
We note that the notion of general point in (2) is that in \cite[Notation
2.2]{kebekus}.
\begin{proof}
  Let $f\colon \mathbb{P}^1 \to X$ be a rational curve passing through $x_0$.
  From Corollary \ref{cor:shviacurves} we immediately find that $f^*TX$ is
  ample, hence
  \begin{equation}\label{eq:ratcurvedirsum}
    f^*TX \simeq \mathcal{O}(d_1) \oplus \mathcal{O}(d_2) \oplus \cdots
    \oplus\mathcal{O}(d_n)
  \end{equation}
  and $d_i \ge 1$ for all $i$.
  In situation $(1)$, we conclude that $X\simeq\bb P^n$ from \cite[V.3.2 Theorem]{kollarrational}.
  \par In situation $(2)$, we have that $d_i \ge 2$ for some $i$ in
  \eqref{eq:ratcurvedirsum} since there is a non-zero natural homomorphism
  $\mathcal{O}(2) \simeq T\mathbb{P}^1 \to {f}^*TX$.
  Thus, $\deg f^*TX = -\deg f^*\omega_X \ge n+1$ for every rational curve
  passing through $x_0$.
  Since $X$ is uniruled by Corollary \ref{cor:uniruled}, we conclude that
  $X\simeq\bb P^n$ from \cite[Corollary 0.4(11)]{cmsb02}.
\end{proof}

See also Corollary \ref{cor:Fanos}.

\par Inspired by Proposition \ref{prop:Fanos}, we conjecture the following:
\begin{conj}
  Let $X$ be a smooth projective variety over an algebraically closed field.
  If there exists $x_0 \in X$ such that $\sh(TX;x_0) > 0$, then $X \simeq \bb
  P^n$.
\end{conj}
We now show the case when $\dim X = 2$.
We start with the following:
\begin{lem}\label{lem:blow-up}
Let $Z\subset X$ be a smooth closed subvariety of a smooth variety. Consider the blow-up cartesian diagram
$$\xymatrix{E\ar@{^{(}->}[r]^{\jmath}\ar[d]_{\pi|_E}\ar@{}[dr]|*={\square}& \widetilde X\ar[d]^{\pi}\\ Z\ar@{^{(}->}[r]_{\imath}& X}.$$
Identify all $x\in X\setminus Z$ with their preimages in $\widetilde X\setminus E$. 
Then for all $x\in X\setminus Z$, 
$$\sh\bigl(\pi^*TX(-E);x\bigr)\leq\sh\bigl(T\widetilde X;x\bigr)\leq \sh\bigl(TX;x\bigr).$$
\end{lem}

\begin{proof}We have a short exact sequence
$0\to\pi^*\Omega_X\to\Omega_{\widetilde X}\to\jmath_*\Omega_{E/Z}\to 0.$
By duality, from the long $\cal Ext$ sequence we extract
\begin{equation}\label{eq:tangent}
0\longrightarrow T\widetilde X\longrightarrow \pi^*TX\longrightarrow\jmath_*T_{E/Z}(E)\longrightarrow 0.
\end{equation}
The second inequality now follows from Lemma \ref{lem:seshadrisequences}. 
We use here that $\sh(\jmath_*T_{E/Z}(E);x)=\infty$, because $x$ is not in the support and $\sh(\pi^*TX;x)$ (computed an $\tilde X$) is the same as $\sh(TX;x)$ (computed on $X$).

For the first inequality, the main ingredient is a short exact sequence
\begin{equation}\label{eq:tangent11}0\longrightarrow\pi^*TX(-E)\longrightarrow T\widetilde X\longrightarrow\jmath_*Q\longrightarrow 0
\end{equation}
Assuming it, we conclude again by Lemma \ref{lem:seshadrisequences}.

From the normal bundle sequence
$0\to TE\to T\widetilde X\rvert_E\to \cal O_E(E)\to 0$
and the relative tangent bundle sequence
$0\to T_{E/Z}\to TE\to\pi|_E^*TZ\to 0,$
we find a bundle $Q$ defined by the sequence
\begin{equation}\label{eq:tangent8}
0\longrightarrow T_{E/Z}\longrightarrow T\widetilde X\rvert_E\longrightarrow Q\longrightarrow 0,
\end{equation}
sitting in 
$0\to \pi|_E^*TZ\to Q\to \cal O_E(E)\to 0.$
 Restrict \eqref{eq:tangent} over $E$, obtaining
$T\widetilde X|_E\to \pi^*TX|_E\to T_{E/Z}(E)\to 0$. 
The first map is the restriction of the differential $d\pi$. 
Its kernel is clearly $T_{E/Z}$, included in $T\widetilde X|_E$ by \eqref{eq:tangent8}.
We obtain another short exact sequence
\begin{equation}\label{eq:tangent10}0\longrightarrow Q\longrightarrow \pi^*TX|_E\longrightarrow T_{E/Z}(E)\longrightarrow 0.
\end{equation}

From the snake lemma for \eqref{eq:tangent} and \eqref{eq:tangent10}, we obtain \eqref{eq:tangent11}.
\end{proof}

\begin{cor}With notation as in the lemma, if $\sh(T\widetilde X;x_0)>0$ for some $x\in X\setminus Z=\widetilde X\setminus E$, then $\sh(TX;x_0)>0$.
\end{cor}

\begin{cor}\label{cor:charpnsurfaces}
  Let $X$ be a smooth projective \emph{surface} over an algebraically closed field. If there exists $x_0\in X$ such that
$\sh(TX;x_0)>0$, then $X\simeq\bb P^2$.
\end{cor}

\begin{proof}Let $E\subset X$ be a smooth curve with negative self-intersection.
Then from the surjection $TX|_E\twoheadrightarrow\cal O_E(E)$ we deduce that $\sh(TX;x)<0$ for
all $x\in E$.

Let $\pi\colon X\to X'$ be a minimal model of $X$ constructed by blowing-down smooth $-1$ curves.
By the previous observation, $x_0$ is not on any of the contracted curves, so it is in the isomorphism locus of $\pi$.
By the previous corollary, $\sh(TX';\pi(x_0))>0$.

The examples at the beginning of the section show that $X'$ is uniruled. 
In the Kodaira classification of minimal surfaces, $X'$ is then either $\bb P^2$, 
or a ruled surface (possibly a Hirzebruch surface).
Remark \ref{rmk:criterionnonpositiveseshadri} excludes ruled surfaces.
Therefore $X'\simeq\bb P^2$.

If $\pi$ is not an isomorphism, then it factors through the blow-up of one point on $\bb P^2$.
This is the Hirzebruch surface $\bb F_1$. 
Apply the previous corollary and Remark \ref{rmk:criterionnonpositiveseshadri} again
to find a contradiction.
\end{proof}

\subsection{Cotangent bundles}

Let $X$ be a smooth projective variety of dimension $n$ over an algebraically closed field.

\begin{ex}[$K_X$ not pseudo-effective]By Lemma \ref{lem:comparetodet}
and Lemma \ref{lem:nonpsefnegative}, we have $\sh(\Omega X;x)\leq\frac 1n\sh(K_X;x)<0$ for all $x\in X$ and $\sh(\Omega_X;x)=-\infty$ for very general $x\in X$. 
\qed
\end{ex}

\begin{rmk}If $Y\subset X$ is a smooth subvariety, then from the surjection $\Omega X|_Y\twoheadrightarrow\Omega Y$ we deduce $\sh(\Omega X;y)\leq\sh(\Omega Y;y)$ for all $y\in Y$.
In particular if $\sh(\Omega Y;y)<0$ then $\sh(\Omega X;y)<0$ by Lemma \ref{lem:shquot}.
\end{rmk}

\begin{ex}[Varieties with rational curves, e.g., $K_X$ not nef] Let $f:\bb P^1\to X$ be a non-constant morphism 
and $x\in f(\bb P^1)$. Then $\sh(\Omega X;x)<0$. 
(We may assume that $f$ is the normalization of its image. Consider the nonzero morphism
$f^*\Omega X\to\Omega \bb P^1=\cal O_{\bb P^1}(-2)$. 
Using Corollary \ref{cor:shviacurves},
we find $\sh(\Omega X;x)\leq\frac{-2}{\mult_xf(\bb P^1)}<0$.)\qed
\end{ex}

\section{Separation of jets}
In this section we give a characterization
of Seshadri constants in terms of separation of jets following \cite[Chapter 5]{laz04}.
First, recall the following:

\begin{defn}
  Let $\mathcal{F}$ be an $\mathcal{O}_X$-module on a projective scheme $X$,
  and fix a closed point $x \in X$ defined by the ideal $\mathfrak{m}_x
  \subseteq \mathcal{O}_X$.
  We say that $\mathcal{F}$ \textsl{separates $s$-jets at $x$} if the
  restriction map
  \[
    H^0(X,\mathcal{F}) \longrightarrow H^0(X,\mathcal{F}/\mathfrak{m}_x^{s+1}\cal F)
  \]
is surjective. With the convention $\frak m_x^0=\cal O_X$, all sheaves
separate $-1$-jets. 
  We denote by $s(\mathcal{F};x)$ the largest integer $s \ge -1$ such that
  $\mathcal{F}$ separates $s$-jets at $x$.
\end{defn}

\begin{rmk}\label{rmk:jetsetquotients}If $\cal F\to\cal G$ is a morphism of quasi-coherent $\cal O_X$-modules, surjective at $x$, then $s(\cal G;x)\geq s(\cal F;x)$ as follows easily by chasing through the commutative diagram
\[
\xymatrix{H^0(X,\cal F)\ar[r]\ar[d]& H^0(X,\cal G)\ar[d]\\
H^0(X,\cal F/\frak m_x^{s+1}\cal F)\ar@{->>}[r]& H^0(X,\cal G/\frak m_x^{s+1}\cal G)
},
\]
where the bottom map is surjective since $\cal F\to\cal G$ is surjective at $x$
and $\operatorname{Spec}\cal O_X/\frak m_x^{s+1}$ is affine. \qed
\end{rmk}
We show the following analogue of \cite[Theorem
5.1.17]{laz04} for higher ranks.
The statement for $x$ a singular point is new even for line bundles.

\begin{thrm}\label{thrm:jetsep}
  Let $\mathcal{V}$ be an ample coherent sheaf on a
  projective scheme $X$, and let $x \in X$ be a closed point.
  Then,
  \[
    \sh(\mathcal{V};x) \le \lim_{k \to \infty} \frac{s(\Sym^k \mathcal{V};x)}{k},
  \]
  and equality holds if $\cal V$ is locally free at $x$.
\end{thrm}

When $\cal V$ is locally free, for any cartesian diagram
\[
  \xymatrix{
    \mathbb{P}(f^*\mathcal{V}) \ar[r]^{f'}\ar[d]_{\rho'}\ar@{}[dr]|*={\square}
    & \mathbb{P}(\mathcal{V}) \ar[d]^{\rho}\\
    Y \ar[r]_f & X
  }
\]
and any $k\geq 0$, the base change map
\begin{equation}\label{eq:basechangeiso}
  f^* \rho_* \mathcal{O}_{\mathbb{P}(\mathcal{V})}(k) \longrightarrow \rho'_*
  f^{\prime*} \mathcal{O}_{\mathbb{P}(\mathcal{V})}(k)
\end{equation}
is an isomorphism. Both terms are $\Sym^kf^*\cal V$. 
When $\cal V$ is an arbitrary coherent sheaf, then the same conclusion
holds for $k$ sufficiently large. 
We will also need the following lemma.

\begin{lem}[cf.\ {\cite[Proof of Lem.\ 3.7]{Ito13}}]\label{lem:ito37}
  Let $X$ be a scheme, and let $\mathcal{F}$ and $\mathcal{G}$ be
  coherent sheaves on $X$ with $s(\cal F;x)\geq 0$ and $s(\cal G;x)\ge0$.
  Then, for every closed point $x \in X$, we have
  \begin{align*}
    s(\mathcal{F};x) &+ s(\mathcal{G};x) \le s(\mathcal{F} \otimes
    \mathcal{G};x).
    \intertext{Furthermore,}
    s(\Sym^m\cal F;x)&+s(\Sym^n\cal F;x)\leq s(\Sym^{m+n}\cal F;x)
  \end{align*}
for all $m,n\geq 0$.
\end{lem}
\begin{proof}
  We first show that a coherent sheaf $\mathcal{F}$ separates $s$-jets if and
  only if
  \begin{equation}\label{eq:mimiplus1}
    H^0(X,\mathfrak{m}_x^i\mathcal{F}) \longrightarrow H^0(X,\mathfrak{m}_x^i\mathcal{F}/\mathfrak{m}_x^{i+1})
  \end{equation}
  is surjective for every $i \in \{0,1,\ldots,s\}$.
  We proceed by induction on $s$.
  If $s = 0$, then there is nothing to show.
  Now suppose $s > 0$.
  By induction and the fact that a coherent sheaf separating $s$-jets also
  separates all lower order jets, it suffices to show that if $\mathcal{F}$
  separates $(s-1)$-jets, then $\mathcal{F}$ separates $s$-jets if and only if
  \eqref{eq:mimiplus1} is surjective for $i = s$.
  Consider the commutative diagram
  \[
    \xymatrix{
      0 \ar[r]
      & \mathfrak{m}_x^s\mathcal{F}  \ar[r]\ar[d]
      & \mathcal{F} \ar[r]\ar[d]
      & \mathcal{F}/\mathfrak{m}_x^s\mathcal{F} \ar[r] \ar@{=}[d]
      & 0\\
      0 \ar[r]
      & \mathfrak{m}_x^s\mathcal{F}/\mathfrak{m}_x^{s+1}\cal F \ar[r]
      & \mathcal{F}/\mathfrak{m}_x^{s+1}\cal F \ar[r]
      & \mathcal{F}/\mathfrak{m}_x^s\mathcal{F} \ar[r]
      & 0
    }
  \]
   Taking global sections, we obtain the diagram
  \[
    \xymatrix{
      0 \ar[r]
      & H^0(X,\mathfrak{m}_x^s\mathcal{F}) \ar[r]\ar[d]
      & H^0(X,\mathcal{F}) \ar[r]\ar[d]
      & H^0(X,\mathcal{F}/\mathfrak{m}_x^s\cal F) \ar@{=}[d]\ar[r]
      & 0\\
      0 \ar[r]
      & H^0(X,\mathfrak{m}_x^s\mathcal{F}/\mathfrak{m}_x^{s+1}) \ar[r]
      & H^0(X,\mathcal{F}/\mathfrak{m}_x^{s+1}\cal F) \ar[r]
      & H^0(X,\mathcal{F}/\mathfrak{m}_x^s\cal F)
      & 
    }
  \]
  where the top row remains exact by the assumption that $\mathcal{F}$ separates
  $(s-1)$-jets.
  By the snake lemma, we see that the left vertical arrow is surjective if and
  only if the middle vertical arrow is surjective, as desired.
  \par We now prove the lemma.
  Suppose $\mathcal{F}$ separates $i$-jets and $\mathcal{G}$ separates $j$-jets.
  We then have the commutative diagram
  \[
    \xymatrix{
      H^0(X,\mathfrak{m}_x^i\mathcal{F}) \otimes H^0(X,\mathfrak{m}_x^j\mathcal{G}) \ar[r] \ar[d]
      & H^0(X,\mathfrak{m}_x^i\mathcal{F}/\mathfrak{m}_x^{i+1}\cal F \otimes
      \mathfrak{m}_x^j\mathcal{G}/\mathfrak{m}_x^{j+1}\cal G)
      \ar[d]\\
      H^0\bigl(X,\mathfrak{m}_x^{i+j}(\mathcal{F} \otimes \mathcal{G})\bigr) \ar[r]
      & H^0\bigl(X, \mathfrak{m}_x^{i+j}(\mathcal{F} \otimes
      \mathcal{G})/\mathfrak{m}_x^{i+j+1}(\mathcal{F} \otimes \mathcal{G})\bigr)
    }
  \]
  Since the top horizontal arrow is surjective by assumption, and the right
  vertical arrow is surjective, essentially by the the surjectivity of 
  \[
    \mathfrak{m}_x^i/\mathfrak{m}_x^{i+1} \otimes
    \mathfrak{m}_x^j/\mathfrak{m}_x^{j+1} \simeq (\mathfrak{m}_x^{i}\otimes\mathfrak m_x^j) \otimes
    \mathcal{O}_X/\mathfrak{m}_x \twoheadrightarrow
    \mathfrak{m}_x^{i+j}/\mathfrak{m}_x^{i+j+1},
  \]
  we see that the composition from the top left corner to the bottom right
  corner is surjective, hence the bottom horizontal arrow is surjective.
  By running through all combinations of integers $i \le s(\mathcal{F};x)$ and
  $j \le s(\mathcal{G};x)$, we see that $s(\mathcal{F};x) + s(\mathcal{G};x) \le
  s(\mathcal{F} \otimes \mathcal{G};x)$ by the argument in the previous
  paragraph.
\vskip.25cm
  The statement on symmetric powers is similar. Use the commutative diagram
\[
  \begin{gathered}
    \xymatrix{
      {\begin{matrix} H^0(X,\mathfrak{m}_x^i\Sym^m\mathcal{F})\\ \otimes\\ H^0(X,\mathfrak{m}_x^j\Sym^n\mathcal{F})\end{matrix}} \ar[r] \ar[d]
      & H^0\left(X,{\begin{matrix}\mathfrak{m}_x^i\Sym^m\mathcal{F}/\mathfrak{m}_x^{i+1}\Sym^m\mathcal{F}\\ \otimes\\  
      \mathfrak{m}_x^j\Sym^n\mathcal{F}/\mathfrak{m}_x^{j+1}\Sym^n\mathcal{F}\end{matrix}}\right)
      \ar[d]\\
      H^0(X,\mathfrak{m}_x^{i+j}\Sym^{m+n}\mathcal{F}) \ar[r]
      & H^0(X,\mathfrak{m}_x^{i+j}\Sym^{m+n}\mathcal{F}/\mathfrak{m}_x^{i+j+1}\Sym^{m+n}\cal F)
    }
  \end{gathered}
    \qedhere
  \]
\end{proof}

\begin{prop}\label{prop:seshadribiggerthanjet}Let $X$ be a projective scheme, and let $\cal V$ be a coherent sheaf on it. Assume that $\cal V$
is locally free around $x$ and $\sh(\cal V;x)\geq 0$. 
Then,
$$\sh(\cal V;x)\geq s(\cal V;x).$$
Moreover, $$\sh(\cal V;x)\geq\limsup_{k\to\infty}\frac{s(\Sym^k\cal V;x)}k.$$
\end{prop}

\begin{proof}
Note that the second statement implies the first by Lemma \ref{lem:ito37},
since the limit supremum is a supremum by Fekete's lemma.
We have natural maps
$\frak m_x^s\subseteq\pi_*\cal O_{\bl_xX}(-sE)$ for all $s\geq 0$. They are equalities if $x$ is smooth,
or if $s$ is sufficiently large. In either case, for all coherent $\cal V$ that
are locally free around $x$, 
they induce isomorphisms
\begin{equation}\label{eq:liftingbound-1}H^0\bigl(X,\frak m_x^s\cal V/\frak m_x^{s+1}\cal V\bigr)\simeq H^0\bigl(\bb P(\cal V(x)),\cal O(1)\bigr)\otimes H^0\bigl(E,\cal
O_E(-sE)\bigr).
\end{equation}
This is because $\frak m_x^s\cal V/\frak m_x^{s+1}\cal V\simeq
\frak m_x^s/\frak m_x^{s+1}\otimes\cal V(x)$ by the fact that $\cal V$ is flat
at $x$, and because 
$\frak m_x^s/\frak m_x^{s+1}=\pi_*\cal O_E(-sE)$, 
under our assumptions on $x$ and $s$.
When $s\geq 1$, these assumptions also imply that $\cal O_E(-sE)$ is very ample on $E$. When $s=0$, it is globally generated.
The same are true of $\cal O_{\bb P(\cal V(x))}(1)\boxtimes\cal O_E(-sE)$,
and its sections generate the pullback to $\rho'^{-1}E$.
This pullback is $\cal O_{\rho^{\prime-1}E}((\pi'^*\xi-s\rho'^*E)|_{\rho'^{-1}E})$.
\par Let $s\coloneqq   s(\cal V;x)$. Assume $s\geq 0$. 
As in the proof of Lemma \ref{lem:ito37}, we have a surjection
\begin{equation}\label{eq:liftingbound0}H^0(X,\frak m_x^s\cal V)\twoheadrightarrow H^0\bigl(X,\frak m_x^s\cal V/\frak
m_x^{s+1}\cal V\bigr).
\end{equation}
When $x$ is smooth or $s$ is large, then
$\pi'^*\xi-s\rho'^*E$ is globally generated along $\rho'^{-1}E$.
For this, in view of \eqref{eq:liftingbound-1} and \eqref{eq:liftingbound0},
it is enough to show that $H^0(X,\frak m_x^s\cal V)$
determine naturally a subspace of 
$H^0\bigl(Y',\pi'^*\cal O_{Y}(1)\otimes\rho'^*\cal O_{\bl_xX}(-sE)\bigr)$.
Consider the commutative diagram
\[
  \xymatrix{
    \frak m_x^s\cal V \ar@{->>}[r]\ar[d]_{\simeq}
    & \frak m_x^s\cal V/\frak m_x^{s+1}\cal V\ar[d]^{\simeq}\\
    \pi_*\bigl(\pi^*\cal V \otimes \cal O_{\bl_xX}(-sE)\bigr) \ar[r]\ar[d]
    & \pi_*\bigl(\pi^*\cal V \otimes \cal O_E(-sE)\bigr)\ar[d]^{\simeq}\\
    \pi_*\bigl(\pi^*\rho_*\cal O_{\bb P(\cal V)}(1) \otimes \cal
    O_{\bl_xX}(-sE)\bigr)
    \ar[r]\ar[d]
    & \pi_*\bigl(\pi^*\rho_*\cal O_{\bb P(\cal V)}(1) \otimes \cal
    O_E(-sE)\bigr)\ar[d]^{\simeq}\\
    \pi_*\bigl(\rho'_*\pi^{\prime*}\cal O_{\bb P(\cal V)}(1) \otimes \cal
    O_{\bl_xX}(-sE)\bigr) \ar[r]\ar[d]_{\simeq}
    & \pi_*\bigl(\rho'_*\pi^{\prime*}\cal O_{\bb P(\cal V)}(1) \otimes \cal
    O_E(-sE)\bigr)\ar[d]^{\simeq}\\
    (\pi \circ \rho')_*\bigl(\pi^{\prime*}\cal O_{\bb P(\cal V)}(1) \otimes
    \rho^{\prime*}\cal O_{\bl_xX}(-sE)\bigr) \ar[r]
    & (\pi \circ \rho')_*\bigl(\pi^{\prime*}\cal O_{\bb P(\cal V)}(1) \otimes
    \rho^{\prime*}\cal O_{E}(-sE)\bigr)
  }
\]
where the top vertical arrows are isomorphisms by the fact that $\cal V$ is
locally free
at $x$, and the vertical arrows in the second row are obtained from the natural map
$V \to \rho_*\cal O_{\bb P(\cal V)}(1)$; the map on
the right is an isomorphism since $\cal V$ is locally free at $x$.
The arrows in the third row are obtained from base change for the cartesian
diagram in Notation \ref{notn:seshnot}, where the right arrow is an isomorphism
by cohomology and base change since $\cal V$ is locally free at $x$, hence $\pi$
is flat around $x$.
The bottom vertical arrows are isomorphisms by the projection formula.
After taking global sections, the bottom horizontal arrow is still surjective by
the commutativity of the diagram.
Thus, since $\pi^{\prime*}\cal O_{\bb P(\cal V)}(1) \otimes
\rho^{\prime*}\cal O_{E}(-sE)$ is globally generated, we see that
$\pi^{\prime*}\cal O_{\bb P(\cal V)}(1) \otimes \rho^{\prime*}\cal
O_{\bl_xX}(-sE)$ is globally generated along $\rho^{\prime-1}(E)$.

Let $C'\in\cal C'_{\cal V,x}$. 
By Proposition \ref{prop:altintsh}, when $s\geq 0$, to show $\sh(\cal V;x)\geq s$, it is enough to prove that 
$$(\pi'^*\xi-s\rho'^*E)\cdot C'\geq 0.$$
Use global generation along $\rho'^{-1}E$ to produce an 
effective divisor in the class $\pi'^*\xi-s\rho'^*E$ that does not pass through $y$, where $y$ is any point of $C'\cap\rho'^{-1}E$.

If $x$ is smooth, the argument above works when $s\geq 0$. When $s=-1$, there is nothing to prove. 

If $x$ is singular, and if $s(\Sym^k\cal V;x)>0$ for some $k$,
then by Lemma \ref{lem:ito37} we have that $s(\Sym^k\cal V;x)$ is arbitrarily large as $k$ grows. 
Repeat the arguments above for all $\Sym^k\cal V$, 
and use the homogeneity of $\sh(-;x)$ from Lemma \ref{lem:homogeneous}.
Assume $s(\Sym^k\cal V;x)\leq 0$ for all $k$. 
If $s(\Sym^k\cal V;x)=0$, then $\Sym^k\cal V$ is globally generated at $x$, therefore $\sh(\Sym^k\cal V;x)\geq 0$ by Example \ref{ex:ggnef}.
By homogeneity, $\sh(\cal V;x)\geq 0$. 
If $s(\Sym^k\cal V;x)=-1$ for all $k$, then there is nothing to prove. 
\end{proof}

\begin{proof}[Proof of Theorem \ref{thrm:jetsep}]
  Write $\varepsilon = \sh(\mathcal{V};x)$ and $s_k = s(\Sym^k \mathcal{V};x)$.
Let $H$ be a very ample divisor on $X$ that separates $1$-jets. 
Since $\cal V$ is ample, $\Sym^k\cal V\otimes\cal O_X(-H)$ is eventually
globally generated by Remark \ref{rmk:amplesheaves}, hence $s_k\geq 1$ for $k$ sufficiently large by
Lemma \ref{lem:ito37}.
 By Proposition \ref{prop:seshadribiggerthanjet},
it is enough to prove
\[
  \varepsilon\leq\lim_{k\to\infty}\frac{s_k}k.
\]

\noindent Note that the limit exists by Fekete's Lemma, since the
sequence $s_k$ is superadditive by Lemma \ref{lem:ito37}.

  Let $0 < \delta \ll 1$ be arbitrary, and fix positive integers $p_0,q_0$
  such that
  \[
    \varepsilon-\delta < \frac{p_0}{q_0} < \varepsilon.
  \]
  Then, $q_0\pi^{\prime*}\xi - p_0\rho^{\prime*}E$ is ample. 
Indeed, the cone generated by 
$\pi'^*\xi$ and $\pi'^*\xi-\varepsilon\rho'^*E$ is contained in the nef cone, and meets the ample cone because $-\rho'^*E$ is $\pi'$-ample and $\xi$ is ample. Consequently, all the classes in its interior are ample. By Fujita's
vanishing theorem, there exists a natural number $m_0$ such that
  \[
    H^1\bigl(Y',
    \mathcal{O}_{Y'}
    \bigl(m(q_0\pi^{\prime*}\xi - p_0\rho^{\prime*}E) + P\bigr) \bigr) = 0
  \]
  for all $m \ge m_0$ and $P$ a nef Cartier divisor on $Y'$, where $Y' =
  \bl_{\bb P(\cal V(x))}\bb P(\cal V)$ as in Notation \ref{notn:seshnot}.
  Now given any integer $k > m_0q_0$, write $k = mq_0 + q_1$ with $0 \le q_1 <
  q_0$.
  Applying the vanishing above for $P = q_1\pi^{\prime*}\xi$, we have that
  \[
    H^1\bigl(Y',
    \mathcal{O}_{Y'}
    \bigl(k\pi^{\prime*}\xi - mp_0\rho^{\prime*}E \bigr) \bigr) = 0.
  \]
  By the Leray spectral sequence \cite[Lemma 5.4.24]{laz04}, this
  cohomology group is isomorphic to
  $H^1\bigl(\mathbb{P}(\mathcal{V}),\mathcal{O}_{\mathbb{P}(\mathcal{V})}(k)
  \otimes \mathcal{I}_{\mathbb{P}(\mathcal{V}_x)}^{mp_0}\bigr)$ for $k \gg 0$
  (which implies $m \gg 0$), even if the point $x$ is singular.
  Now for $k \gg 0$ (which implies $m \gg 0$), the right vertical arrow in the
  commutative diagram
  \[
    \xymatrix{
      H^0(X,\Sym^k\mathcal{V}) \ar[r]\ar@{=}[d]
      & H^0(X,\Sym^k\mathcal{V} \otimes \mathcal{O}_X/\mathfrak{m}_x^{mp_0})
      \ar[d]^\simeq\\
      H^0\bigl(\mathbb{P}(\mathcal{V}),
      \mathcal{O}_{\mathbb{P}(\mathcal{V})}(k)\bigr) \ar[r]
      & H^0\bigl(\mathbb{P}(\mathcal{V}),
      \mathcal{O}_{\mathbb{P}(\mathcal{V})}(k) \otimes
      \mathcal{O}_{\mathbb{P}(\mathcal{V})}/
      \mathcal{I}_{\mathbb{P}(\mathcal{V}_x)}^{mp_0}\bigr)
    }
  \]
  is an isomorphism by the base change isomorphism
  \eqref{eq:basechangeiso} applied to 
  $\Spec(\mathcal{O}_X/\mathfrak{m}_x^{mp_0}) \subseteq X$.
  The bottom arrow is therefore surjective for $m \gg 0$.
  Thus, $\Sym^k\mathcal{V}$ separates $mp_0-1$ jets, and
  \[
    \frac{s_k}{k} \ge \frac{mp_0 - 1}{k} \ge \frac{mp_0 - 1}{(m+1)q_0} =
    \frac{m}{m+1} \cdot \frac{p_0}{q_0} - \frac{1}{(m+1)q_0} >
    \frac{m}{m+1}(\varepsilon - \delta) - \frac{1}{(m+1)q_0}.
  \]
  Taking limit infima as $k \to \infty$, we have $m \to \infty$, hence
  \[
    \liminf_{k \to \infty} \frac{s_k}{k} \ge \varepsilon - \delta,
  \]
  and since $\delta$ was arbitrary, the conclusion follows.
\end{proof}

It is known that lower bounds on Seshadri constants of big and nef invertible sheaves $\cal L$ lead to lower bounds on the jet separation of adjoint bundles $\omega_X\otimes\cal L$. See \cite[Proposition 6.8]{dem}. In this direction, Hacon proves

\begin{thrm}[{\cite[Theorem 1.7]{hacon}}]\label{thrm:hacon17}Let $\cal V$ be an ample locally free
  sheaf of finite rank $r$ on a complex projective manifold of dimension $n$. Let
  $\beta\in\bb Q_+$ such that $\pi^*\cal V^{\vee}\langle\beta\xi\rangle$ is
  ample. Set
  \begin{equation}\label{eq:Mhacon}
    M\coloneqq  \min_{0\leq i\leq n-1}\left[\frac 1{{\binom{n+r-i}{r}}^\frac 1{n-i}}\cdot\frac 1{n-i}\right].
  \end{equation}
Then for any integer $\lambda>{n\beta}/{M}$, the locally free sheaf $\omega_X\otimes\Sym^{\lambda}\cal V\otimes\det\cal V$ is generated by global
sections at all very general points $x\in X$.
\end{thrm}
\cite[Theorem $5.2.2.1'$]{Cataldo} is a result of similar flavor.
Hacon's global generation result is a corollary of his lower bounds on Seshadri constants \cite[Theorem 1.5.a.i]{hacon}.
These generalize the line bundle case of \cite{ekl95}.
Theorem \ref{thrm:hacon17} is then an instance of the following jet separation bound:

\begin{prop}\label{prop:jetsepbound}Let $X$ be a complex projective manifold of dimension $n$, and let $\cal V$ be an ample (or $\cal O_{\bb P(\cal V)}(1)$ is only big and nef) locally free sheaf of finite rank $r\geq 1$ on $X$.
If $p\geq 0$ is such that $\sh(\cal V;x)>\frac{n+s}{p+r}$, then $\omega_X\otimes\Sym^p\cal V\otimes\det\cal V$ separates $s$-jets at $x$. 

In particular, if $\sh(\cal V;x)>\frac nr$ for all $x\in X$, then $\omega_X\otimes\det\cal V$ is globally generated.
\end{prop}
Compare with \cite[Proposition 6.8]{dem}.
A relative version of this argument yields a higher-rank analogue of
\cite[Theorem 2.2]{Cataldoaif}; see Theorem \ref{thrm:dmthmaanalogue}.

\begin{proof}We follow the proof of the Griffiths vanishing result in \cite[Theorem 7.3.1]{laz042}. We prove that 
$H^1\bigl(X, \omega_X\otimes\Sym^p\cal V\otimes\det\cal V\otimes\frak m_x^{s+1}\bigr)=0$.
This is equivalent to 
$$H^1\bigl(\bb P(\pi^*\cal V), \omega_{\bb P(\pi^*\cal V)}\otimes\cal O_{\bb P(\pi^*\cal V)}(p+r)\otimes\rho'^*\cal O(-(n+s)E)\bigr)=0.$$
By Remark \ref{rmk:easyseshadri}, we know that $(p+r)\pi'^*\xi-(n+s)\rho'^*E$ is nef. 
It is also big as it is a positive combination between the big divisor $\pi'^*\xi$ and the nef (so pseudo-effective) $\pi'^*\xi-\sh(\cal V;x)\rho'^*E$. 
The conclusion follows from the Kawamata--Viehweg vanishing theorem.
\end{proof}

\section{Base loci}
\label{section:loci}
Building on ideas of Nakamaye, \cite[Remark 6.5]{ELMNP} proves that if $D$ is a big and nef divisor on a smooth projective variety,
then the Seshadri constants of $D$ determine the augmented base locus:
\[\bb B_+(D)=\{x\in X\st \sh(D;x)=0\}.\]
We aim to prove a generalization to sheaves.
Let $\cal V$ be a coherent sheaf on a protective scheme $X$.
\cite[Definition 2.1]{bundleloci} defines the \emph{base locus} of $\cal V$ as 
\[\bs(\cal V)\coloneqq  \{x\in X\st H^0(X,\cal V)\to \cal V(x)\text{ is not surjective}\}.\]
With notation as in Notation \ref{notn:seshnot}, when $\cal V$ is locally free,
the relation with the base locus of $\cal O_{\bb P(\cal V)}(1)$ is given by
\[\rho\bigl(\bs\bigl(\cal O_{\bb P(\cal V)}(1)\bigr)\bigr)=\bs(\cal V).\]
\begin{rmk}\label{rmk:baselocusbundle}
More precisely, $\bs(\cal V)=\Supp\cal Q$ and, $\bs(\cal O_{\bb P(\cal V)}(1))\subseteq \bb P(\cal Q)$, with equality when $\cal V$ is locally free. 
Here $\cal Q$ determined by
\[\pushQED{\qed}H^0(X;\cal V)\otimes\cal O_X\stackrel{{\rm ev}}{\longrightarrow} \cal
V\longrightarrow \cal Q\longrightarrow 0.\qedhere\popQED\]
\end{rmk}
\noindent The \emph{stable base locus} of a coherent sheaf $\cal V$ is
\[\bb B(\cal V)\coloneqq  \bigcap_{k\geq 1}\bs(\Sym^k\cal V).\]
Let $gg(\cal V)\coloneqq   X\setminus\bs(\cal V)$ be the globally generated locus of $\cal V$.
From
\begin{equation}\label{eq:gglocus}
gg(\cal V)\subseteq gg(\Sym^m\cal V)\subseteq gg(\Sym^k\Sym^m\cal V)\subseteq gg(\Sym^{km}\cal V),
\end{equation}
we deduce that 
\[\bb B(\cal V)\subseteq\bb B(\Sym^m\cal V)\]
for all $m\geq 1$.
While the inclusion
\[\rho\bigl(\bb B\bigl(\cal O_{\bb P(\cal V)}(1)\bigr)\bigr)\subseteq\bb B(\cal
V)\]
is easy to prove (see \cite[p.\ 233]{bundleloci}), equality may fail, 
even when $\cal V$ is locally free (see \cite[Example 3.2]{miurb}).
However, equality does hold if one allows perturbations.
\begin{defn}[{\cite[Definition 2.4]{bundleloci}}]
  The \emph{augmented base locus} of a coherent sheaf $\cal V$ is
\[\bb B_+(\cal V)\coloneqq  \bigcap_{k\geq 0}\bb B\bigl(\Sym^k\cal V\otimes\cal
O_X(-H)\bigr),\]
where $H$ is any ample divisor on $X$.
\end{defn}
To show that the definition is independent of $H$, we prove the following:
\begin{lem}\label{lem:bplususebaselocus}Let $\cal V$ be a coherent sheaf on a projective scheme over an algebraically closed field. 
Then for all ample divisors $H$, 
\[\bigcap_{k\geq 0}\bb B\bigl(\Sym^k\cal V\otimes\cal
O_X(-H)\bigr)=\bigcap_{k\geq 0}\bs\bigl(\Sym^k\cal V\otimes\cal O_X(-H)\bigr).\]
If $H$ is ample and globally generated, then the intersection on the right-hand
side stabilizes to $\bs\bigl(\Sym^k\cal V\otimes\cal O_X(-H)\bigr)$ for all sufficiently divisible $k$.
\end{lem}
\begin{proof}If $x\in gg\bigl(\Sym^k\cal V\otimes\cal O_X(-H)\bigr)$, then clearly 
$x\not\in\bb B\bigl(\Sym^k\cal V\otimes\cal O_X(-H)\bigr)$.
Conversely, if $x\in gg\bigl(\Sym^m\bigl(\Sym^k\cal V\otimes\cal O_X(-H)\bigr)\bigr)$, then 
$x\in gg\bigl(\Sym^{pmk}\cal V\otimes\cal O_X(-pmH)\bigr)$ for all $p\geq 1$. 
For large $p$, so that $(pm-1)H$ is globally generated, we obtain $x\in gg\bigl(\Sym^{pmk}\cal V\otimes\cal O_X(-H)\bigr)$.

When $H$ is ample and globally generated, then 
$\bs\bigl(\Sym^k\cal V\otimes\cal O_X(-H)\bigr)\supseteq 
\bs\bigl(\Sym^{mk}\cal V\otimes\cal O_X(-H)\bigr)$ for all $k,m\geq 1$.
Conclude by noetherianity.
\end{proof}

We now deduce independence of $H$ in the definition of $\bb B_+(\cal V)$. 
\begin{cor}[cf.\ {\cite[Remark 2.5.1]{bundleloci}}]\label{cor:bundleloci251}Let $\cal V$ be as above, and let $A$ and $H$ be ample divisors. Then,
\[\bigcap_{k\geq 0}\bs\bigl(\Sym^k\cal V\otimes\cal O_X(-H)\bigr)=\bigcap_{k\geq
0}\bs\bigl(\Sym^k\cal V\otimes\cal O_X(-A)\bigr).\]
In particular, the definition of $\bb B_+(\cal V)$ is independent of the choice
of ample divisor $H$.
\end{cor}
\begin{proof}If $x\in gg\bigl(\Sym^k\cal V\otimes\cal O_X(-H)\bigr)$, then $x\in gg\bigl(\Sym^{mk}\cal V\otimes\cal O_X(-mH)\bigr)$
for all $m\geq 1$. In particular if $m$ is large enough so that $mH-A$ is globally generated, then 
$x\in gg\bigl(\Sym^{mk}\cal V\otimes\cal O_X(-A)\bigr)$. This proves one inclusion. The other one follows by symmetry.
\par The last statement follows from the above and Lemma
\ref{lem:bplususebaselocus}.
\end{proof}

The relation between $\bb B_+(\cal V)$ and $\bb B_+(\cal O_{\bb P(\cal V)}(1))$ is given by the following:

\begin{prop}\label{prop:b+push}Let $\cal V$ be a coherent sheaf on a projective scheme over an algebraically closed field.
  Then,
\[\bb B_+(\cal V)\supseteq\rho\bigl(\bb B_+\bigl(\cal O_{\bb P(\cal V)}(1)\bigr)\bigr).\]
Equality holds when intersecting with the open locally free locus of $\cal V$.
\end{prop}
\begin{proof}
\noindent \cite[Proposition 3.2]{bundleloci} proves that when $\cal V$ is locally free on complex projective manifolds, 
and the proof in general is essentially the same.
Let $H$ be a very ample divisor on $X$ such that $\cal V(H)$ is globally generated.
We obtain a surjection $H^0(X,\ \cal V(H))\otimes\cal O_X(H)\twoheadrightarrow \cal V(2H)$, which shows that 
$A\coloneqq  \cal O_{\bb P(\cal V)}(1)\otimes\rho^*\cal O_X(2H)$ is very ample on $\bb P(\cal V)$. 

Assume $x\in gg\bigl(\Sym^k\cal V\otimes\cal O_X(-H)\bigr)$. Then 
$\rho^{-1}\{x\}=\bb P(\cal V(x))\subseteq gg\bigl(\cal O_{\bb P(\cal V)}(2k)\otimes\rho^*\cal O_X(-2H)\bigr)$.
We have $\cal O_{\bb P(\cal V)}(2k)\otimes\rho^*\cal O_X(-2H)=\cal O_{\bb P(\cal V)}(2k+1)\otimes A^{\vee}$. 
These show the ``$\supseteq$'' inclusion. 

Assume now $\bb P(\cal V(x))\subseteq\bigcup_{k\geq 0}gg\bigl(\cal O_{\bb P(\cal V)}(k)\otimes A^{\vee}\bigr)$.
Since $A$ is very ample, we have inclusions 
$gg\bigl(\cal O_{\bb P(\cal V)}(k)\otimes A^{\vee}\bigr)\subseteq gg\bigl(\cal O_{\bb P(\cal V)}(mk)\otimes A^{\vee}\bigr)$ 
for all $m\geq 1$ and all $k\geq 0$. We deduce that $\cal O_{\bb P(\cal V)}(k)\otimes A^{\vee}=\cal O_{\bb P(\cal V)}(k-1)\otimes\rho^*\cal O_X(-2H)$ is globally generated
along $\bb P(\cal V(x))$ for sufficiently divisible $k$. 
Pushing forward to $X$, since $\rho_*\cal O_{\bb P(\cal V)}(k)=\Sym^k\cal V$ for $k$ large enough, we find that the canonical map
\[H^0\bigl(X,\ \Sym^{k-1}\cal V\otimes\cal O_X(-2H)\bigr)\otimes\rho_*\cal O_{\bb P(\cal V)}\to\Sym^{k-1}\cal V\otimes\cal O_X(-2H)\]
is surjective at $x$. If $x$ is in the locally free locus of $\cal V$, then the natural map 
$\cal O_X\to\rho_*\cal O_{\bb P(\cal V)}$ is an isomorphism around $x$, hence $x\in gg\bigl(\Sym^{k-1}\cal V\otimes\cal O_X(-2H)\bigr)$. 
\end{proof}

\begin{rmk}\cite{ELMNP} and \cite{birkar} define augmented base loci
of $\bb R$-Cartier $\bb R$-divisors. If $\cal V$ is locally free,
one can use the result above to define $\bb B_+(\cal V\langle\lambda\rangle)\coloneqq  \rho(\bb B_+(\xi+\rho^*\lambda))$.
\end{rmk}

We start relating $\bb B_+(\cal V)$ to Seshadri constants. 

\begin{lem}\label{lem:bplussesh}Let $\cal V$ be a coherent sheaf on a projective scheme $X$. If $x\not\in\bb B_+(\cal V)$, then $\sh(\cal V;x)>0$.
\end{lem}
\begin{proof}The assumptions imply that for every ample Cartier divisor $H$ on $X$ there exists $k>0$ such that $\Sym^{mk}\cal V\otimes\cal O_X(-mH)$ is globally generated at $x$
for sufficiently large $m$. Then $\sh(\Sym^{mk}\cal V\otimes\cal O_X(-mH);x)\geq 0$ by Example \ref{ex:ggnef}.
By Lemmas \ref{lem:homogeneous} and \ref{lem:tensorproducts} 
we get
\begin{align*}\sh(\cal V;x)=\frac 1{mk}\sh\bigl(\Sym^{mk}\cal V;x\bigr)&\geq\frac 1{mk}\sh\bigl(\Sym^{mk}\cal V\otimes\cal O_X(-mH);x\bigr)+\frac 1{mk}\sh(\cal O_X(mH);x)\\
  &\geq \frac 1k\sh(\cal O_X(H);x)>0.\qedhere
\end{align*}
\end{proof}

\begin{cor}\label{cor:Fanos}
  Let $X$ be a smooth projective variety over a field of characteristic zero.
If $\bb B_+(TX)\subsetneq X$, then $X\simeq\bb P^n$.
\end{cor}
\begin{proof}Lemma \ref{lem:bplussesh} implies $\sh(TX;x)>0$ for $x$ a general point on $X$.
Now use Proposition \ref{prop:Fanos}(2).
\end{proof}

\begin{defn}[{see \cite[Theorem 6.4]{bundleloci}}]\label{def:vbig} A sheaf $\cal
  V$ is called \emph{V-big}\footnote{``V'' stands for Viehweg.} if $\bb B_+(\cal
  V)\neq X$.
\end{defn}

\noindent \cite[Examples 1.7 and 1.8]{Jab09} shows that this is usually stronger than asking for $\cal O_{\bb P(\cal V)}(1)$ to be big, even when $\cal V$ is locally free.
See also \cite[Remark 6.6]{bundleloci}.

The main result of this section is the following:
\begin{prop}\label{prop:baseseshadrivanish}
  Let $\cal V$ be a locally free sheaf of finite rank on a 
  projective scheme $X$ over an algebraically closed field, 
and suppose that $\cal V$ is nef. 
Then,
  \[\bb B_+(\cal V)=\bigl\{x\in X\st \sh(\cal V;x)=0\bigr\}.\]
If $\cal V$ is only a coherent sheaf (but still nef), and $x$ is in the locally free locus of $\cal V$, then $x\in\bb B_+(\cal V)$ if and only if 
$\sh(\cal V;x)=0$. 
\end{prop}

\begin{proof}In view of Lemma \ref{lem:bplussesh}, it is enough to justify the ``$\subseteq$'' inclusion.
Let $x\in\bb B_+(\cal V)$ such that $\cal V$ is locally free around $x$. 
By Proposition \ref{prop:b+push}, there exists $y\in\bb P(\cal V(x))$ such that $y\in\bb B_+(\xi)$.
Since $\xi$ is nef, \cite{birkar} proves that there exists a
subvariety $Z\subseteq \bb P(\cal V)$ through $y$ such
that $\xi^{\dim Z}\cdot Z=0$. 
By \cite[Proposition 5.1.9]{laz04}, we deduce $\sh(\xi;y)=0$.
Conclude by Remark \ref{rmk:compareless1}.
\end{proof}

We obtain an immediate improvement of Theorem \ref{seshadriample}.

\begin{cor}Let $X$ be a projective scheme. 
Let $\cal V$ be a nef locally free sheaf of finite rank on $X$.
Then, $\cal V$ is ample if and only if $\sh(\cal V;x)>0$ for all $x\in X$.
\end{cor}

The following lemma will be used in the proof of
Theorem \ref{thrm:dmthmaanalogue}.

\begin{lem}\label{lem:bplussjets}
  Let $X$ be a projective scheme, and let $\mathcal{V}$ be a coherent sheaf on $X$.
  If $x \notin \mathbb{B}_+(\mathcal{V})$ is a closed point, then for every
  coherent sheaf $\mathcal{F}$ on $X$ and every integer $s \ge 0$, the sheaf
  $\mathcal{F} \otimes_{\mathcal{O}_X} \Sym^m \mathcal{V}$
  separates $s$-jets at $x$ for all $m$ sufficiently large.
\end{lem}
\begin{proof}
Let $H$ be a very ample divisor on $X$ that separates $s$-jets at $x$.
Since $x\not\in\bb B_+(\cal V)$, there exists $m\geq 1$ 
such that $x\in gg\big(\Sym^m\cal V\otimes\cal O_X(-H)\bigr).$

Let $n_0$ be sufficiently large so that 
$\cal F\otimes\Sym^r\cal V\otimes\cal O_X(nH)$ 
separates $s$-jets at $x$ for all $0\leq r<m$ and all $n\geq n_0$.
Such $n_0$ exists by Lemma \ref{lem:ito37}.

For $M\geq mn_0$, write $M=mq+r$ with $0\leq r<m$ and $q\geq n_0$.
Then $\cal F\otimes\Sym^M\cal V$ is a quotient of 
$\cal F\otimes\Sym^r\cal V\otimes\Sym^q\Sym^m\cal V=
\bigl(\cal F\otimes\Sym^r\cal V\otimes\cal O_X(qH)\bigr)
\otimes\Sym^q(\Sym^m\cal V\otimes\cal O_X(-H))$.
Conclude by Lemma \ref{lem:ito37}.
\end{proof}

\begin{cor}[Stability of augmented base loci]\label{cor:b+stable}
With assumptions as in the lemma, let $H$ be an ample divisor on $X$.
Then for all sufficiently large $m$, we have 
\[\bb B_+(\cal V)=\bb B_+\bigl(\Sym^m\cal V\otimes\cal O_X(-H)\bigr).\]
\end{cor}
\begin{proof}We have 
$gg\Bigl(\Sym^n\bigl(\Sym^m\cal V\otimes\cal O_X(-H)\bigr)\otimes\cal
O_X(-H)\Bigr)\subseteq 
gg\bigl(\Sym^{nm}\cal V\otimes\cal O_X(-(n+1)H)\bigr)$. This proves the ``$\subseteq$'' inclusion for all $m\geq 1$.
Assume $x\not\in\bb B_+(\cal V)$. By Lemma \ref{lem:bplussjets}
there exists $m_x\geq 1$ such that
$\Sym^m\cal V\otimes\cal O_X(-2H)$ is globally generated at $x$ 
for all $m\geq m_x$.
In particular $x\not\in\bb B_+\bigl(\Sym^m\cal V\otimes\cal O_X(-H)\bigr)$
for $m\geq m_x$. 
The constant $m_x$ can be made independent of $x$ by noetherianity,
since $gg\bigl(\Sym^m\cal V\otimes\cal O_X(-2H)\bigr)$ and $X\setminus\bb B_+(\cal V)$ are open.
\end{proof}

\section{Direct images of pluricanonical sheaves}
In this section, we prove the following analogue of \cite[Theorem
A]{DuttaMurayama} for higher-rank bundles and for higher-order jets, in the
spirit of a relative Fujita-type conjecture of Popa and Schnell \cite[Conjecture
1.3]{popaschnell}.
\begin{thrm}\label{thrm:dmthmaanalogue}
  Let $f\colon Y \to X$ be a surjective morphism of complex projective
  varieties, where $X$ is of dimension $n$.
  Let $(Y,\Delta)$ be a log canonical $\mathbb{R}$-pair and let $\mathcal{V}$ be
  a locally free sheaf of finite rank $r \ge 1$ on $X$ such that $\cal O_{\bb
  P(\cal V)}(1)$ is big and nef.
  Consider a Cartier divisor $P$ on $Y$ such that $P \sim_{\mathbb{R}}
  k(K_Y+\Delta)$ for some integer $k \ge 1$, and consider a general smooth
  closed point $x \in X \setminus \mathbb{B}_+(\mathcal{V})$.
  If we have
  \begin{equation}\label{eq:dmthmaanalogueineq}
    \sh(\mathcal{V};x) > k \cdot \frac{n+s}{m+k(r-1)+1},
  \end{equation}
  then the sheaf
  \begin{equation}\label{eq:toshowsepsjets}
    f_*\mathcal{O}_Y(P) \otimes_{\mathcal{O}_X} \Sym^m\mathcal{V}
    \otimes_{\mathcal{O}_X} (\det \mathcal{V})^{\otimes k}
  \end{equation}
  separates $s$-jets at $x$.
  \par In particular, if $X$ is smooth, $\mathcal{V}$ is ample, and $\beta>0$ is
  such that $\pi^*\mathcal{V}^\vee\langle\beta\xi\rangle$ is ample, then with $M$ as in \eqref{eq:Mhacon}, for
  every integer
  \[
    \lambda > k \cdot \biggl( \frac{\beta}{M}(n+s) - (r-1) \biggr)-1
  \]
  the sheaf $f_*\mathcal{O}_Y(P) \otimes
  \Sym^\lambda\cal V \otimes (\det \mathcal{V})^{\otimes k}$ 
separates $s$-jets at all general points $x \in X$.
\end{thrm}
Note that by Proposition \ref{prop:baseseshadrivanish}, the condition $x \notin
\mathbb{B}_+(\mathcal{V})$ follows from the condition on $\sh(\mathcal{V};x)$
in \eqref{eq:dmthmaanalogueineq}.
This condition also implies $\cal V$ is V-big in the sense of Definition
\ref{def:vbig}.
We also note that the last statement follows in the same way as in \cite[Theorem
1.7]{hacon}, using the lower bound for Seshadri constants in \cite[Theorem
1.5.a.i]{hacon}, hence it suffices to show the first statement.
Finally, our statement has ``general'' instead of ``very general'' since
separating $s$-jets is an open condition.
\begin{proof}
By applying Lemma \ref{lem:bplussjets} to $\mathcal{F} =
  f_*\mathcal{O}_Y(P) \otimes (\det \mathcal{V})^{\otimes k}$, there exists a
  smallest positive integer $m_0$ such that the sheaf \eqref{eq:toshowsepsjets}
  separates $s$-jets at $x$ for $m = m_0$.
We will prove that
the sheaf \eqref{eq:toshowsepsjets} separates $s$-jets at $x$
for a suitable choice of a general point $x$, if
\begin{equation}\label{eq:psneatidea}
      \sh(\mathcal{V};x) > \frac{n+s}{m+r-\frac{k-1}{k}m_0}.
\end{equation}
The choice of the general point $x$ will be detailed momentarily,
but first we explain how the conclusion of the theorem follows from
\eqref{eq:psneatidea}.
This inequality is equivalent to
  \[
    m > \frac{n+s}{\sh(\mathcal{V};x)} + \frac{k-1}{k}m_0 - r,
  \]
  and by the minimality of $m_0$, we see that
  \[
    m_0 \le \biggl\lfloor\frac{n+s}{\sh(\mathcal{V};x)} + \frac{k-1}{k}m_0 - r
    \biggr\rfloor + 1 \le \frac{n+s}{\sh(\mathcal{V};x)} + \frac{k-1}{k}m_0 - r
    + 1.
  \]
  Rearranging this inequality yields
  \[
    m_0 \le k \cdot \biggl( \frac{n+s}{\sh(\mathcal{V};x)} -r + 1\biggr),
  \]
  and substituting this upper bound for $m_0$ into the inequality
  for $m$ above, we see that the sheaf \eqref{eq:toshowsepsjets} separates
  $s$-jets at $x$ if
  \[
    m > \frac{n+s}{\sh(\mathcal{V};x)} + (k-1)\cdot\biggl(
    \frac{n+s}{\sh(\mathcal{V};x)} -r + 1\biggr) - r
    = k\cdot \frac{n+s}{\sh(\mathcal{V};x)} - k(r-1) - 1
  \]
  which is equivalent to the inequality \eqref{eq:dmthmaanalogueineq}.
  This idea was inspired by the proof of \cite[Theorem 1.7]{popaschnell}.  

  We now explain the choice of the general point $x$.
  Following Steps 0 and 1 in the proof of \cite[Theorem A]{DuttaMurayama}, we
  may assume that $Y$ is smooth, that $\Delta$ has simple normal crossings
  support and coefficients in $(0,1]$, and that the image of the adjunction
  morphism
  \[
    f^*f_*\mathcal{O}_Y(P) \longrightarrow \mathcal{O}_Y(P)
  \]
  is of the form $\mathcal{O}_Y(P-G)$ for a divisor $G$ such that $\Delta+G$ has
  simple normal crossings support.
  We will show that under these assumptions, the sheaf \eqref{eq:toshowsepsjets}
  separates $s$-jets at all smooth closed points $x \in X \setminus
  \mathbb{B}_+(\mathcal{V})$ satisfying \eqref{eq:dmthmaanalogueineq}, such that
  $f$ is smooth at $x$ and such that the fiber $Y_x \coloneqq  
  f^{-1}(x)$ over $x$ intersects each component of $\Delta$ transversely.
  \begin{step}
    Reduction to the case $k = 1$ for a suitable pair.
  \end{step}
By assumption on $m_0$, we know that the sheaf \eqref{eq:toshowsepsjets}
  separates $s$-jets at $x$ for $m = m_0$, and in particular, is globally
  generated at $x$.  
This implies that the sheaf
  \[
    \mathcal{O}_Y(P-G) \otimes \Sym^{m_0} f^*\mathcal{V} \otimes (\det
    f^*\mathcal{V})^{\otimes k}
  \]
  is globally generated along $Y_x$.
  By pulling back along the bundle map
\[
\rho_Y\colon \bb P_Y(f^*\cal V)\longrightarrow Y,
\]   
and using the $m_0$th symmetric power of the
  tautological quotient map, the 
  invertible sheaf
  \[
    \mathcal{O}_{\mathbb{P}_Y(f^*\mathcal{V})} \bigl( \rho_Y^*(P-G)\bigr)
    \otimes \mathcal{O}_{\mathbb{P}_{Y}(f^*\mathcal{V})}(m_0) \otimes \bigl(\det
    (f \circ \rho_Y)^*\mathcal{V}\bigr)^{\otimes k}
  \]
  on $\mathbb{P}_{Y}(f^*\mathcal{V})$ is globally generated along
  $\rho_Y^{-1}(Y_x)$.
  Now let $c_1(f^*\mathcal{V})$ denote the divisor class of the determinant of
  $f^*\mathcal{V}$ on $Y$, and let $\eta=c_1(\cal O_{\mathbb{P}_{Y}(f^*\mathcal{V})}(1))$.
 Switching to divisor notation, 
  \begin{align*}
    \rho_Y^*(P-G) + k\,\rho_Y^*c_1(f^*\mathcal{V}) + m_0\eta
    &\sim_{\mathbb{R}} \rho_Y^*(k\Delta-G) + k\,\rho_Y^*K_Y +
    k\,\rho_Y^*c_1(f^*\mathcal{V}) + m_0\eta\\
    &\sim_{\mathbb{R}} \rho_Y^*(k\Delta-G) +
    k\,K_{\mathbb{P}_{Y}(f^*\mathcal{V})} + (m_0+kr) \eta.
  \end{align*}
  By Bertini's theorem, we can therefore choose a general divisor
  \[
    \mathfrak{D} \in \bigl\lvert \rho_Y^*(P-G) + k\,\rho_Y^*c_1(f^*\mathcal{V}) + m_0\eta \bigr\rvert
  \]
  that is smooth along $\rho_Y^{-1}(Y_x)$, and intersects both
  $\rho_Y^{-1}(Y_x)$ and the supports of $\rho_Y^*\Delta$ and $\rho_Y^*G$
  transversely (see \cite[Lemma 4.1.11]{laz04}) in a neighborhood of $\rho_Y^{-1}(Y_x)$.
  We then have
  \[
    k\bigl(K_{\mathbb{P}_{Y}(f^*\mathcal{V})}+\rho_Y^*\Delta\bigr)
    \sim_{\mathbb{R}} K_{\mathbb{P}_{Y}(f^*\mathcal{V})}+\rho_Y^*\Delta +
    \frac{k-1}{k}\mathfrak{D} + \frac{k-1}{k}\rho_Y^*G -
    \frac{k-1}{k}(m_0+kr)\eta.
  \]
  \par We now want to rewrite the right-hand side as the sum of a log canonical
  divisor coming from a log canonical pair on $\mathbb{P}_{Y}(f^*\mathcal{V})$
  and a multiple of $\eta$.
  Since $\Delta + \frac{k-1}{k}G$ may have some coefficients greater than one,
  we first adjust the coefficients of $\Delta$ and $G$.
  Applying \cite[Lemma 2.18]{DuttaMurayama} to $c = \frac{k-1}{k}$, there
  exists an effective Cartier $\bb Z$-divisor $G' \preceq G$ such that
  \[
    \Delta' \coloneqq   \Delta + \frac{k-1}{k}G - G'
  \]
  is effective with simple normal crossings support, with components
  intersecting $Y_x$ transversely, and with coefficients in $(0,1]$.
  Since $\rho_Y$ is a smooth morphism, the pullback $\rho_Y^*\Delta'$ also has
  these same properties on $\mathbb{P}_Y(f^*\mathcal{V})$.
  We then have
  \begin{equation}\label{eq:decomponpyfstarv}
    \begin{aligned}
      \rho_Y^*\bigl(P + k\,c_1(f^*\mathcal{V}) - G'\bigr) &\sim_{\mathbb{R}}
      k\bigl(K_{\mathbb{P}_{Y}(f^*\mathcal{V})}+r\eta+\rho_Y^*\Delta\bigr)-
      \rho_Y^*G'\\
      &\sim_{\mathbb{R}} K_{\mathbb{P}_{Y}(f^*\mathcal{V})}+\rho_Y^*\Delta' +
      \frac{k-1}{k}\mathfrak{D} + \biggl(r - \frac{k-1}{k}m_0\biggr)\eta.
    \end{aligned}
  \end{equation}
  This $\mathbb{R}$-linear equivalence will be used to reduce the case $k > 1$
  for the pair $(Y,\Delta)$ to the case $k = 1$ for the pair
  $(\mathbb{P}_{Y}(f^*\mathcal{V}),\rho_Y^*\Delta' +
  \frac{k-1}{k}\mathfrak{D})$.
  \begin{step}
    Replacing $\mathfrak{D}$ with a divisor with simple normal crossings
    support.
  \end{step}
  Let $\mu\colon Z \to \mathbb{P}_Y(f^*\mathcal{V})$ be a common log resolution
  for $\mathfrak{D}$ and $(\mathbb{P}_Y(f^*\mathcal{V}),\rho_Y^*\Delta')$.  
  Note that we can choose $\mu$ to be an isomorphism along $\rho_Y^{-1}(Y_x)$,
  since $\mathfrak{D}$ and $\rho_Y^*\Delta'$ intersect transversely and have
  simple normal crossings support in a neighborhood of $\rho_Y^{-1}(Y_x)$.
  We can then write
  \[
    \mu^*\mathfrak{D} = \mathfrak{D}_1 + F, \qquad (\rho_Y\circ\mu)^*\Delta' =
    \mu_*^{-1}(\rho_Y^*\Delta') + F_1,
  \]
  where $\mathfrak{D}_1$ is a smooth divisor intersecting $(\rho_Y \circ
  \mu)^{-1}(Y_x)$ transversely and $F,F_1$ are supported away from $(\rho_Y \circ
  \mu)^{-1}(Y_x)$.
  Define
  \begin{gather*}
    F' \coloneqq   \biggl\lfloor \frac{k-1}{k} F + F_1 \biggr\rfloor, \qquad
    \widetilde{\Delta} \coloneqq   (\rho_Y\circ\mu)^*\Delta' +
    \frac{k-1}{k}\mu^*\mathfrak{D} - F',\\
    \widetilde{P} \coloneqq   (\rho_Y\circ\mu)^*\bigl(P +
    k\,c_1(f^*\mathcal{V})\bigr) + K_{Z/\mathbb{P}_Y(f^*\mathcal{V})}.
  \end{gather*}
  Note that $\widetilde{\Delta}$ has simple normal crossings support and
  coefficients in $(0,1]$ by assumption on the log resolution $\mu$ and by the
  definition of $F'$, and also has components intersecting $(\rho_Y \circ
  \mu)^{-1}(Y_x)$ transversely.
  Pulling back the decomposition in \eqref{eq:decomponpyfstarv} via $\mu$ and
  adding $K_{Z/\mathbb{P}_Y(f^*\mathcal{V})} - F'$ yields
  \begin{equation}\label{eq:frakdnowhassnc}
        \widetilde{P} - (\rho_Y\circ\mu)^*G' - F'
      \sim_{\mathbb{R}} K_Z + \widetilde{\Delta} +
      \biggl( r-\frac{k-1}{k}m_0\biggr)\mu^*\eta.
    \end{equation}
  \begin{step}\label{step:dmthmaanalogue4}
    To show the sheaf \eqref{eq:toshowsepsjets} separates
    $s$-jets at $x$, it suffices to show that the sheaf
    \begin{equation}\label{eq:toshowsepsjetsmodified}
      (f \circ \rho_Y \circ \mu)_*\mathcal{O}_Z\bigl(\widetilde{P} -
      (\rho_Y\circ\mu)^*G' - F' + m\mu^*\eta\bigr)
    \end{equation}
    separates $s$-jets at $x$.
  \end{step}
  Consider the commutative diagram
  \[
    \scalebox{0.75}{\xymatrix{
      H^0\bigl(X,(f \circ \rho_Y \circ \mu)_*\mathcal{O}_Z\bigl(\widetilde{P} -
      (\rho_Y\circ\mu)^*G' - F' + m\mu^*\eta\bigr)\bigr) \ar[r]
      \ar@{^(->}[d]
      & H^0\bigl(X,(f \circ \rho_Y \circ \mu)_*\mathcal{O}_Z\bigl(\widetilde{P} -
      (\rho_Y\circ\mu)^*G' - F' + m\mu^*\eta\bigr) \otimes
      \frac{\mathcal{O}_X}{\mathfrak{m}_x^{s+1}}\bigr)\ar[d]^{\simeq}\\
      H^0\bigl(X,(f \circ \rho_Y \circ \mu)_*\mathcal{O}_Z\bigl(\widetilde{P} -
      (\rho_Y\circ\mu)^*G'+m\mu^*\eta\bigr)\bigr) \ar[r]
      \ar[d]_{\simeq}
      & H^0\bigl(X,(f \circ \rho_Y \circ \mu)_*\mathcal{O}_Z\bigl(\widetilde{P} -
      (\rho_Y\circ\mu)^*G'+m\mu^*\eta\bigr) \otimes
      \frac{\mathcal{O}_X}{\mathfrak{m}_x^{s+1}}\bigr)\ar[d]^{\simeq}\\
      H^0\bigl(X,f_*\mathcal{O}_Y(P-G') \otimes \Sym^m\mathcal{V}
      \otimes (\det\mathcal{V})^{\otimes k} \bigr) \ar[r] \ar[d]_{\simeq}
      & H^0\bigl(X,f_*\mathcal{O}_Y(P-G') \otimes \Sym^m\mathcal{V}  \otimes
      (\det\mathcal{V})^{\otimes k} \otimes
      \frac{\mathcal{O}_X}{\mathfrak{m}_x^{s+1}}\bigr) \ar[d]^{\simeq}\\
      H^0\bigl(X,f_*\mathcal{O}_Y(P) \otimes \Sym^m\mathcal{V}
      \otimes (\det\mathcal{V})^{\otimes k} \bigr) \ar[r]
      & H^0\bigl(X,f_*\mathcal{O}_Y(P) \otimes \Sym^m\mathcal{V}  \otimes
      (\det\mathcal{V})^{\otimes k} \otimes
      \frac{\mathcal{O}_X}{\mathfrak{m}_x^{s+1}}\bigr)
    }}
  \]
  where the top right isomorphism holds since $F'$ is supported away from
  $(\rho_Y \circ \mu)^{-1}(Y_x)$.
  The vertical isomorphisms in the middle row follow from the
  projection formula, the fact that $K_{Z/\mathbb{P}_Y(f^*\mathcal{V})}$ is
  $\mu$-exceptional, and the fact that
  $\mathbb{R}\rho_{Y*}\mathcal{O}_{\mathbb{P}_Y(f^*\mathcal{V})}(m)$ is
  quasi-isomorphic to $\Sym^mf^*\mathcal{V}$ for $m \ge 0$.
  The vertical isomorphisms in the bottom row follow from \cite[Lemma
  2.17]{DuttaMurayama}.
  If the top horizontal arrow is surjective, then the commutativity of the
  diagram implies that the bottom horizontal arrow is also surjective, i.e., the
  sheaf in \eqref{eq:toshowsepsjets} separates $s$-jets at $x$.
  \begin{step}\label{step:dmthmaanalogue5}
    The sheaf \eqref{eq:toshowsepsjetsmodified} separates $s$-jets at $x$ if
    \[
      \sh(\mathcal{V};x) > \frac{n+s}{m+r-\frac{k-1}{k}m_0}.
    \]
  \end{step}
   
Consider the commutative diagram
  \[
    \xymatrix{
Z' \ar[r]^{\pi_Z}\ar[d]_{\mu'}\ar@{}[dr]|*={\square} & Z\ar[d]^{\mu}\\
      \mathbb{P}_{Y'}(\mathcal{W}) \ar[r]^-{\pi_Y'}\ar@{}[dr]|*={\square}
      \ar[d]_{\rho_{Y'}} & \mathbb{P}_Y(f^*\mathcal{V}) \ar[d]^{\rho_Y}\\
      Y' \ar[r]^{\pi_Y} \ar[d]_{f'}\ar@{}[dr]|*={\square} & Y \ar[d]^{f}\\
      X' \ar[r]^{\pi} & X
    }
  \]
  with cartesian squares,
  where $X' = \bl_xX$, where $Y' = \bl_{Y_x}Y$, and
  $\mathcal{W} = (f \circ
  \pi_Y)^*\mathcal{V} = (\pi \circ f')^*\mathcal{V}$.
  The bottom square is cartesian since $f$ is flat at $x$.
Since $\rho_Y$ is smooth and therefore flat, we also have
$\bb P_{Y'}(\cal W)=\bl_{\rho_Y^{-1}Y_x}\bb P_Y(f^*\cal V)$.
In the top square, $\pi_Z$ is the blow-up of $Z$ along $(\rho_Y \circ \mu)^{-1}(Y_x)$ since $\mu$ is an isomorphism over $\rho_Y^{-1}(Y_x)$. 
Consider the commutative diagram
  \begin{equation}\label{eq:extensionh0}
    \mathclap{\begin{gathered}
      \scalebox{0.75}{\xymatrix{
          H^0\bigl(Z',\pi_Z^*\mathcal{O}_Z
            \bigl(\widetilde{P} - (\rho_Y\circ\mu)^*G'
          - F' + m\mu^*\eta\bigr)\bigr) \ar[r]
          & H^0\bigl(Z',\pi_Z^*\mathcal{O}_Z
	\bigl(\widetilde{P} - (\rho_Y\circ\mu)^*G' - F' +
          m\mu^*\eta\bigr)\bigr\rvert_{(t+1)\mu^{\prime*}E}\bigr)\\
          H^0\bigl(Z,\mathcal{O}_Z
            \bigl(\widetilde{P} - (\rho_Y\circ\mu)^*G'
          - F' + m\mu^*\eta\bigr)\bigr) \ar[r] \ar[u]^{\simeq}
          & H^0\bigl(Z,\mathcal{O}_Z\bigl(\widetilde{P} - (\rho_Y\circ\mu)^*G'
          - F' + m\mu^*\eta\bigr)/\mathcal{I}_{(\rho_Y \circ
          \mu)^{-1}(Y_x)}^{t+1}\bigr) \ar[u]_{\simeq}\\
          H^0\bigl(X,(f \circ \rho_Y \circ \mu)_*\mathcal{O}_Z\bigl(\widetilde{P} -
          (\rho_Y\circ\mu)^*G' - F' + m\mu^*\eta\bigr)\bigr) \ar[r]
          \ar[u]^{\simeq}
          & H^0\bigl(X,(f \circ \rho_Y \circ \mu)_*\mathcal{O}_Z\bigl(\widetilde{P} -
            (\rho_Y\circ\mu)^*G' - F' + m\mu^*\eta\bigr) \otimes
          \frac{\mathcal{O}_X}{\mathfrak{m}_x^{t+1}}\bigr) \ar[u]_{\alpha_t(x)}
      }}
    \end{gathered}}
  \end{equation}
  where the vertical arrows in the top row are isomorphisms by the fact that
  $\pi_Z$ is the blow-up along the smooth subscheme
  $(\rho_Y \circ \mu)^{-1}(Y_x) \subseteq Z$; see \cite[Lemma 4.3.16]{laz04}.
  We will show that the top horizontal arrow is surjective for $t = 0$ and $t
  = s$.
  The $t = 0$ statement will show that $\alpha_t(x)$ is surjective by the
  commutativity of the diagram, hence an isomorphism for all
  $t$ by cohomology and base change \cite[Corollary 8.3.11]{illusiefga}, using
  the flatness of $f$ at $x$.
  The surjectivity of the top horizontal arrow for $t = s$ will then show
  that the sheaf \eqref{eq:toshowsepsjetsmodified}
  separates $s$-jets at $x$.
  \par Choose a sufficiently small positive rational number $\delta$ such that
  \[
    \sh(\mathcal{V};x) > \frac{n+s+\delta}{m+r-\frac{k-1}{k}m_0}.
  \]
  Let $D$ denote the exceptional divisor for the blow-up
  $\mathbb{P}_{X'}(\pi^*\mathcal{V}) \to \mathbb{P}_X(\mathcal{V})$ along
  $\mathbb{P}(\mathcal{V}(x))$, let $\xi$ denote the Serre class on
  $\mathbb{P}_{X'}(\pi^*\mathcal{V})$, and let 
  $E$ denote the exceptional divisor of the blow-up $\pi'_Y$.
  The $\mathbb{Q}$-divisor
  \begin{equation}\label{eq:nearenddiv}
    \begin{aligned}
      \MoveEqLeft[7]\mu^{\prime*}\biggl(\biggl( m + r - \frac{k-1}{k}m_0
      \biggr)\pi_Y^{\prime*}\eta - (n+t+\delta)E\biggr)\\
      &= (\rho_{Y'} \circ \mu')^*\biggl( \biggl( m + r - \frac{k-1}{k}m_0 \biggr)\xi
      - (n+t+\delta)D\biggr)
    \end{aligned}
  \end{equation}
  is big and nef for $t\in\{0,s\}$ by assumption on $\sh(\mathcal{V};x)$ and
  Remark \ref{rmk:easyseshadri}(c). 
  By the definition of $\sh(\mathcal{V};x)$ and \cite[Remark 6.5]{ELMNP}, the
  stable base locus of the divisor \eqref{eq:nearenddiv} is disjoint from
  $\mu^{\prime*}E$ (cf.\ the proof of \cite[Lemma 3.3]{DuttaMurayama}).
  By Bertini's theorem, for $\ell$ a sufficiently large and divisible integer, 
  we can therefore choose a general divisor
  \[
    \mathfrak{E} \in \biggl\lvert \ell \biggl(\mu^{\prime*}\biggl(\biggl( m + r -
    \frac{k-1}{k}m_0 \biggr)\pi_Y^{\prime*}\eta - (n+t+\delta)E\biggr)\biggr)
    \biggr\rvert
  \]
  that is smooth along $\mu^{\prime*}E$, and intersects every component of the
  support of $\pi^*_Z\widetilde{\Delta}$ transversely in a neighborhood of
  $\mu^{\prime*}E$.
  \par Choose a common log resolution $\nu\colon \widetilde{Z}' \to Z'$ for
  $\mathfrak{E}$ and $(Z',\pi^*_Z\widetilde{\Delta})$ that is an isomorphism
  along $\mu^{\prime*}E$.
  We then write
  \[
    \nu^*\mathfrak{E} = \mathfrak{E}_1 + B, \qquad (\pi_Z \circ
    \nu)^*\widetilde{\Delta} = \nu_*^{-1}\pi^*_Z\widetilde{\Delta} + B_1
  \]
  where $\mathfrak{E}_1$ is a smooth prime divisor intersecting $(\mu' \circ
  \nu)^*E$ transversely and $B,B_1$ are supported away from $(\mu' \circ
  \nu)^*E$.
  Define
  \begin{gather*}
    B' \coloneqq   \biggl\lfloor \frac{1}{\ell} B + B_1 \biggr\rfloor, \qquad
    \Gamma \coloneqq   (\pi_Z \circ \nu)^*\widetilde{\Delta} +
    \frac{1}{\ell}\nu^*\mathfrak{E} - B' + \delta(\mu' \circ \nu)^*E,\\
    Q \coloneqq   (\pi_Z \circ \nu)^*\widetilde{P} + K_{\widetilde{Z}'/Z'},
  \end{gather*}
  where we note that $\Gamma$ has simple normal crossings support and
  coefficients in $(0,1]$, since $\pi^*_Z\widetilde{\Delta}$ has simple normal
  crossings support and coefficients in $(0,1]$ by the condition that
  $\widetilde{\Delta}$ has components intersecting $(\rho_Y \circ
  \mu)^{-1}(Y_x)$ transversely; see \cite[Corollary 6.7.2]{fulton84}.
  By the $\mathbb{R}$-linear equivalence \eqref{eq:frakdnowhassnc}, we
  have that
  \begin{align*}
    \MoveEqLeft[3]\pi_Z^*\widetilde{P} - (\rho_Y\circ\mu
    \circ \pi_Z)^*G' - \pi_Z^*F' + m(\mu\circ\pi_Z)^*\eta -
    (t+1)\mu^{\prime*}E\\
    &\sim_{\mathbb{R}} K_{Z'} + \pi_Z^*\widetilde{\Delta} + \delta\mu^{\prime*}E
    + \frac{1}{\ell}\mathfrak{E}
  \end{align*}
  where we use the fact that $\pi_Z$ is the blow-up along the smooth subvariety
  $(\rho_Y \circ \mu)^{-1}(Y_x)$ of codimension $n$.
  Pulling back along $\nu$ and adding $K_{\widetilde{Z}'/Z} - B'$, we obtain
  \begin{equation}\label{eq:fujinoequiv}
    \begin{aligned}
      \MoveEqLeft[3]Q - (\rho_Y\circ\mu \circ \pi_Z \circ \nu)^*G' - (\pi_Z
      \circ \nu)^*F' - B' + m(\mu\circ\pi_Z \circ \nu)^*\eta - (t+1)(\mu' \circ
      \nu)^*E\\
      &\sim_{\mathbb{R}} K_{\widetilde{Z}'} + \Gamma.
    \end{aligned}
  \end{equation}
  Since $B'$ is supported away from $(\mu'\circ\nu)^*E$ and
  $K_{\widetilde{Z}'/Z}$ is $\nu$-exceptional, an argument similar to Step
  \ref{step:dmthmaanalogue4} shows that 
  to show the surjectivity of the top horizontal arrow in
  \eqref{eq:extensionh0}, it suffices to show that the morphism
  \begin{multline*}
    H^1\bigl(\widetilde{Z}',\mathcal{O}_{\widetilde{Z}'}\bigl(
    Q - (\rho_Y\circ\mu \circ \pi_Z \circ \nu)^*G' - (\pi_Z \circ
    \nu)^*F' - B' + m(\mu\circ\pi_Z \circ \nu)^*\eta -
    (t+1)\mu^{\prime*}E\bigr)\bigr)\\
    \longrightarrow
    H^1\bigl(\widetilde{Z}',\mathcal{O}_{\widetilde{Z}'}\bigl(
    Q - (\rho_Y\circ\mu \circ \pi_Z \circ \nu)^*G' - (\pi_Z \circ
    \nu)^*F' - B' + m(\mu\circ\pi_Z \circ \nu)^*\eta\bigr)\bigr)
  \end{multline*}
  is injective.
  This injectivity follows from Fujino's Koll\'ar-type injectivity theorem
  \cite[Theorem 5.4.1]{fujinommp} by using the $\mathbb{R}$-linear equivalence
  \eqref{eq:fujinoequiv} and the fact that $\Gamma$ contains $(\mu' \circ
  \nu)^*E$ in its support.
  \par The argument above works for $t = 0$ or $t = s$, hence the
  sheaf \eqref{eq:toshowsepsjetsmodified} separates $s$-jets at $x$.
  \end{proof}
Specializing to the case when $\mathcal{V}$ is an invertible sheaf, we obtain
the following version of \cite[Theorem A]{DuttaMurayama} for higher-order jets
using the lower bound on Seshadri constants in \cite{ekl95}.
This also gives a generic version of \cite[Corollary 2.7]{ShentuZhang} for big
and nef line bundles that are not necessarily globally generated, albeit with
weaker bounds.
\begin{cor}\label{cor:dmthmaanalogue}
  Let $f\colon Y \to X$ be a surjective morphism of complex projective
  varieties, where $X$ is of dimension $n$.
  Let $(Y,\Delta)$ be a log canonical $\mathbb{R}$-pair and let $\mathcal{L}$ be
  a big and nef invertible sheaf on $X$.
  Consider a Cartier divisor $P$ on $Y$ such that $P \sim_{\mathbb{R}}
  k(K_Y+\Delta)$ for some integer $k \ge 1$.
  Then, the sheaf
  \[
    f_*\mathcal{O}_Y(P) \otimes_{\mathcal{O}_X} \mathcal{L}^{\otimes \ell}
  \]
  separates $s$-jets at all general points $x \in X$ for all $\ell \ge k
  (n(n+s)+1)$.
\end{cor}
Just as in the case when $s = 0$, one can replace the lower bound $\ell \ge k
(n(n+s)+1)$ with the lower bound $\ell \ge k((n-1)(n+s)+1)$ when
$X$ is smooth of dimension at most three and $\mathcal{L}$ is ample; see
\cite[Remark 5.2]{DuttaMurayama}.

\bibliographystyle{amsalpha}
\bibliography{Seshadri}

\end{document}